\pdfoutput=1
\documentclass[11pt,a4paper]{article}
\usepackage[utf8]{inputenc}
\usepackage[T1]{fontenc}
\usepackage[english]{babel}
\usepackage[final]{microtype}
\usepackage[a4paper,margin=2.5cm]{geometry}
\usepackage{setspace}
\usepackage{csquotes}
\usepackage{titlesec}
\titleformat{\section}{\large\bfseries}{\thesection}{0.6em}{}
\titleformat{\subsection}{\normalsize\bfseries}{\thesubsection}{0.6em}{}
\titleformat{\subsubsection}{\normalsize\itshape}{\thesubsubsection}{0.6em}{}

\usepackage{fancyhdr}
\usepackage{amsmath,amssymb,amsfonts}
\usepackage{amsthm}
\numberwithin{equation}{section}
\newtheorem{theorem}{Theorem}[section]
\newtheorem{lemma}[theorem]{Lemma}
\newtheorem{corollary}[theorem]{Corollary}
\newtheorem{proposition}[theorem]{Proposition}

\newtheorem{remark}[theorem]{Remark}

\usepackage{graphicx}
\graphicspath{{figures/}}
\usepackage{caption}
\usepackage{subcaption}
\usepackage{booktabs}
\usepackage{siunitx}
\usepackage{float}
\usepackage{placeins}
\usepackage{needspace}
\usepackage{diagbox}
\usepackage{listings}
\usepackage[dvipsnames]{xcolor}
\usepackage[hidelinks]{hyperref}
\hypersetup{
	pdftitle={Delay Placement Governs Feasibility and Hopf Bifurcation in a Pollution-Based Tourism Model},
	pdfauthor={Taylan Şengül and Bünyamin Kurtkaya},
	pdfsubject={Delay placement governs feasibility and Hopf bifurcation in a pollution-based tourism model},
	pdfkeywords={Hopf bifurcation, delay differential equations, time delay, positivity, delay placement, sustainable tourism},
	colorlinks=true,
	linkcolor=black,
	citecolor=MidnightBlue,
	urlcolor=MidnightBlue
}
\usepackage[nameinlink,noabbrev]{cleveref}
\usepackage{orcidlink}

\usepackage{tikz}
\usetikzlibrary{arrows.meta,positioning,calc,patterns}
\colorlet{slotone}{BrickRed}
\colorlet{slottwo}{Goldenrod!75!black}
\colorlet{slotthree}{MidnightBlue}

\usepackage[backend=biber,style=numeric, sorting=none, maxbibnames=99,uniquename=false,giveninits=true]{biblatex}
\newcommand{\keywords}[1]{\par\small\textbf{Keywords:}~#1}

\usepackage{authblk}

\title{Delay Placement Governs Feasibility and Hopf Bifurcation in a Pollution-Based Tourism Model}
\author[1]{Taylan Şengül\,\orcidlink{0000-0003-4841-7326}}
\author[1,2]{Bünyamin Kurtkaya\,\orcidlink{0009-0003-0290-5719}}
\affil[1]{Department of Mathematics, Marmara University, Göztepe Kampüsü, Fahrettin Kerim Gökay Caddesi, 34722 Kadıköy, İstanbul, Türkiye}
\affil[2]{Department of Basic Sciences, Turkish Naval Academy, National Defence University, 34940 Tuzla, İstanbul, Türkiye}
\affil[ ]{\textnormal{\texttt{taylan.sengul@marmara.edu.tr}\quad\texttt{bunyamin.kurtkaya@msu.edu.tr} (corresponding author)}}

\date{\today}

\begin{document}
\maketitle

\begin{abstract}
	We compare eight variants of a three-variable pollution-based tourism model obtained by assigning the same discrete delay to different terms.
	The delay may enter three slots: the visitor factor in the deterrence term, pollution perception, or visitation-driven growth and reinvestment.
	We call a placement feasible when every nonnegative initial history produces a solution that exists for all future time and remains nonnegative.
	For every $\tau>0$, feasibility holds exactly when the visitor factor in the deterrence term is current, which occurs in four of the eight placements.
	In three of these four feasible placements, the positive equilibrium is locally asymptotically stable for every delay.
	In the fourth---the perception-lag model, with delayed pollution perception but current visitation---a conjugate pair of characteristic roots crosses the imaginary axis at a finite critical delay for every set of positive parameters.
	The equilibrium then loses stability and does not restabilise as the delay increases.
	A center-manifold reduction and numerical parameter sweep show that both supercritical and subcritical Hopf bifurcations can occur.
	The visitor-delayed placement, closest to earlier delayed-tourism models, can undergo a local Hopf bifurcation under additional conditions, but it does not preserve nonnegativity for all initial histories.
\end{abstract}


\section{Introduction}

Protected areas are expected to reconcile ecological conservation with diverse social and economic objectives \cite{watson2014}.
The coupling between visitation, pollution, and capital naturally involves delays: visitors respond to previously perceived pollution, visitation-driven growth and reinvestment may lag visitation, and the deterrence term may depend on past rather than current visitor levels.
Whether the resulting dynamics remain physically feasible, however, may depend on which mechanism uses past information.

We build on the delayed environmental-protection expenditure model proposed by Russu \cite{russu2009}:

\begin{equation}\label{eq:russu}
	\begin{aligned}
		 & \dot V(t)=m_1 E(t)+m_2 K(t)-a V^2(t),                  \\
		 & \dot E(t)=r(\bar{P}-E(t))-b V(t-\tau)+c \eta V(t-\tau), \\
		 & \dot K(t)=(1-\eta) V(t-\tau)-\delta K(t).
	\end{aligned}
\end{equation}

Here $V$, $E$, and $K$ denote visitors, environmental quality, and capital stock, while $\tau$ is the delay.
Following Becker \cite{becker1982} and Cazzavillan and Musu \cite{cazzavillan2001}, Russu defines $E=\bar{P}-P$, where $P$ is the current pollution stock and $\bar{P}$ its maximum tolerable level.

Caraballo et al. \cite{caraballo2019} corrected Russu's equilibrium conditions and showed that \eqref{eq:russu} does not preserve nonnegativity: some solutions with nonnegative initial histories become negative.
Their proposed fix prevents solutions from crossing the environmental boundary $E=0$ by multiplying the delayed deterrence by the current environmental quality $E$.
Although they left the modified system unanalysed, we show that their modification while preserving nonnegativity also prevents Hopf bifurcation (Appendix~\ref{app:caraballo}).
He et al. \cite{he2022,he2023} added a second delay in the visitor equation and analysed a Hopf bifurcation.
Their delayed crowding term, however, can still push the visitor variable below zero from the boundary $V=0$ (Remark~\ref{rem:he-noninv}).

Delay-induced Hopf bifurcations are well studied in tourism, bioeconomic, pollution, and toxicant delay differential equation (DDE) models \cite{lacitignola2010,monfared2016,kaslik2020,liu2024,ardeuan2025,adelpini2025,dipesh2026}.
In models with nonnegative state variables, however, locating a Hopf bifurcation is not enough: one must first verify that nonnegative histories generate nonnegative solutions.
If a negative interaction uses a past value of a variable that is currently zero, the interaction need not vanish and may push that variable below zero \cite{seifert1976,martin1990}.
The trajectory then leaves the physically meaningful nonnegative state space \cite{dipesh2026,singh2025}.
Related work on delayed logistic equations shows that oscillations depend on how the delay enters the feedback, not only on its magnitude \cite{baker2020}.

Motivated by these observations, we compare a family of models for visitors $V$, pollution stock $P$, and capital $K$, obtained by assigning delays to three mechanisms in Russu's framework:
\begin{equation}\label{eq:family}
	\begin{aligned}
		 & \dot V=v_{1}V(t-\tau_{3})-v_{2}V(t-\tau_{1})P(t-\tau_{2})+v_{3}K(t), \\
		 & \dot P=p_{1}V(t)-p_{2}P(t),                                      \\
		 & \dot K=k_{1}V(t-\tau_{3})-k_{3}K(t),
	\end{aligned}
\end{equation}
where $\tau_{1}$ is assigned to the visitor factor in the deterrence term, $\tau_{2}$ to pollution perception, and $\tau_{3}$ to visitation-driven growth and reinvestment (Figure~\ref{fig:schematic-mech}).
The growth and reinvestment terms share $\tau_{3}$ because both represent delayed expansion in response to visitation.
This is a modelling simplification, not a claim that the lags coincide; separating them would give sixteen placements.
Equation~\eqref{eq:family} is a mechanism-level reformulation rather than an algebraic transformation of \eqref{eq:russu}.
It replaces Russu's additive environmental-quality and crowding effects in the visitor equation by the per-capita growth rate $v_1-v_2P$, retains capital-driven visitor growth, and models pollution directly through an emission--decay balance.
Unlike Russu, we impose no finite upper bound on $P$.
Larger $P$ means more degradation, so the natural decay term appears as $-p_{2}P$.
Accordingly, the effect of environmental degradation on visitation enters through the deterrence term $-v_{2}VP$.
The model parameters are defined in Table~\ref{table1}.

\begin{table}[H]
	\centering
	\begin{tabular}{@{}cp{10cm}@{}}
		\toprule
		Parameter & Description                                 \\
		\midrule
		$v_{1}$   & Visitation-driven growth coefficient.       \\
		$v_{2}$   & Visitor--pollution deterrence coefficient.  \\
		$v_{3}$   & Capital-driven visitor growth coefficient. \\
		$p_{1}$   & Per-capita pollution emission coefficient. \\
		$p_{2}$   & Natural pollution decay rate.               \\
		$k_{1}$   & Visitation-to-capital conversion coefficient. \\
		$k_{3}$   & Capital depreciation rate.                  \\
		\bottomrule
	\end{tabular}
	\caption{Model parameters of the pollution-based tourism family \eqref{eq:family}.
		All are strictly positive.}
	\label{table1}
\end{table}

We focus on the common-delay placements satisfying
\[
	(\tau_{1},\tau_{2},\tau_{3})\in\{0,\tau\}^3.
\]
Each triple is a \emph{placement}; the eight placements are the corners of the \emph{placement cube} $\{0,\tau\}^3$, and selecting one gives the corresponding model (Figure~\ref{fig:schematic-cube}).
The placement names used below are shorthand for their distinguishing delay pattern; the triple gives the complete specification.
For example, the instantaneous-deterrence placement still delays growth and reinvestment, while the pollution-delayed placement delays both pollution perception and growth and reinvestment.
This restriction isolates placement effects while leaving the full multi-delay problem open.

\begin{figure}[t]
	\centering
	\begin{subfigure}[b]{0.46\textwidth}
		\centering
		\resizebox{\linewidth}{!}{%
			\begin{tikzpicture}[
					font=\footnotesize,
					>={Stealth[length=5pt]},
					var/.style={circle,draw,thick,minimum size=9.5mm,fill=#1!6,draw=#1!70!black},
					lag/.style={rectangle,rounded corners=2pt,draw=#1,fill=#1!12,text=#1,inner sep=1.5pt,font=\scriptsize\bfseries},
					elab/.style={font=\scriptsize,inner sep=1pt},
				]
				\node[var=OliveGreen] (P) at (-2.55,0) {$P$};
				\node[var=black] (V) at (0,0) {$V$};
				\node[var=black] (K) at (2.55,0) {$K$};
				\draw[->,thick] (V) to[bend right=35]
				node[elab,above=2pt] {$+\,p_1$} (P);
				\draw[-{Bar[width=5.5pt]},very thick,slottwo!85!black] (P) to[bend right=35]
				node[lag=slottwo,pos=0.3] {$\tau_2$}
				node[lag=slotone,pos=0.65] {$\tau_1$}
				node[elab,below=4pt,pos=0.4,text=black] {$-\,v_2 VP$}
				(V);
				\node[text=slottwo] at (-1.28,0) {$\circlearrowright\,-$};
				\draw[->,thick] (K) to[bend left=35]
				node[elab,below=2pt,pos=0.5] {$+\,v_3$} (V);
				\draw[->,thick] (V) to[bend left=35]
				node[lag=slotthree,pos=0.4] {$\tau_3$}
				node[elab,above=3pt,pos=0.78] {$+\,k_1$} (K);
				\draw[->,thick] (V.115) .. controls +(115:1.2) and +(65:1.2) .. (V.65)
				node[lag=slotthree,pos=0.5,above=7pt] {$\tau_3$};
				\node[elab] at (0.95,1.05) {$+\,v_1$};
				\draw[->,gray] (P.145) .. controls +(145:0.7) and +(-175:0.7) .. (P.-175);
				\node[elab,gray] at (-3.28,0.72) {$-\,p_2$};
				\draw[->,gray] (K.35) .. controls +(35:0.7) and +(-5:0.7) .. (K.-5);
				\node[elab,gray] at (3.28,0.72) {$-\,k_3$};
			\end{tikzpicture}}
		\caption{Feedback structure and delay slots}
		\label{fig:schematic-mech}
	\end{subfigure}
	\hfill
	\begin{subfigure}[b]{0.52\textwidth}
		\centering
		\resizebox{\linewidth}{!}{%
			\begin{tikzpicture}[>={Stealth[length=4pt]},font=\scriptsize]
				\def\s{3.0}
				\def\dx{1.3}
				\def\dy{1.0}
				\coordinate (f00) at (0,0);        
				\coordinate (f10) at (\s,0);       
				\coordinate (f01) at (0,\s);       
				\coordinate (f11) at (\s,\s);      
				\coordinate (b00) at (\dx,\dy);          
				\coordinate (b10) at (\s+\dx,\dy);       
				\coordinate (b01) at (\dx,\s+\dy);       
				\coordinate (b11) at (\s+\dx,\s+\dy);    
				\fill[pattern=north east lines,pattern color=red!45] (b00)--(b10)--(b11)--(b01)--cycle;
				\draw[red!50,thin] (b00)--(b10)--(b11)--(b01)--cycle;
				\draw[gray!70,thin,dashed] (f00)--(b00) (b00)--(b10) (b00)--(b01);
				\draw[gray,thin] (f10)--(b10) (f01)--(b01) (f11)--(b11);
				\fill[green!55!black,opacity=0.10] (f00)--(f10)--(f11)--(f01)--cycle;
				\draw[black!60,thin] (f00)--(f10)--(f11)--(f01)--cycle;
				\fill[green!45!black] (f00) circle (2.1pt);
				\fill[green!45!black] (f01) circle (2.1pt);
				\fill[green!45!black] (f11) circle (2.1pt);
				\node[Goldenrod!80!black,font=\normalsize] at (f10) {$\bigstar$};
				\node[red!70!black,font=\small] at (b01) {$\blacktriangle$};
				\draw[gray,fill=white] (b00) circle (2.1pt);
				\draw[gray,fill=white] (b10) circle (2.1pt);
				\draw[gray,fill=white] (b11) circle (2.1pt);
				\node[below left=0pt,align=right] at (f00) {delay-free\\$(0,0,0)$};
				\node[below right=0pt,align=left] at (f10) {perception-lag\\$(0,\tau,0)$};
				\node[above left=0pt,align=right] at (f01) {instantaneous-\\deterrence\\$(0,0,\tau)$};
				\node[below left=0pt,align=right,fill=white,fill opacity=0.85,text opacity=1,inner sep=1.5pt] at (f11) {pollution-delayed\\$(0,\tau,\tau)$};
				\node[gray,right=2pt,align=left,fill=white,fill opacity=0.85,text opacity=1,inner sep=1.5pt] at (b00) {visitor-only\\$(\tau,0,0)$};
				\node[gray,right=2pt,align=left] at (b10) {deterrence-\\delayed\\$(\tau,\tau,0)$};
				\node[above right=0pt,align=left] at (b01) {visitor-delayed\\$(\tau,0,\tau)$};
				\node[gray,above right=0pt,align=left] at (b11) {fully-delayed\\$(\tau,\tau,\tau)$};
				\coordinate (o) at (-1.6,1.6);
				\draw[->] (o)--++(0.85,0) node[right] {$\tau_2$};
				\draw[->] (o)--++(0,0.85) node[above] {$\tau_3$};
				\draw[->] (o)--++(0.5,0.36) node[above right,inner sep=0.5pt] {$\tau_1$};
				\begin{scope}[shift={(-0.3,-1.35)},font=\scriptsize]
					\fill[green!45!black] (0,0) circle (2.1pt);
					\node[right=3pt] at (0,0) {delay-independently stable};
					\node[Goldenrod!80!black,font=\normalsize] at (0,-0.45) {$\bigstar$};
					\node[right=3pt] at (0,-0.45) {Hopf at $\tau_{0}$};
					\node[red!70!black,font=\small] at (0,-0.90) {$\blacktriangle$};
					\node[right=3pt] at (0,-0.90) {infeasible; local Hopf under conditions};
					\draw[gray,fill=white] (0,-1.35) circle (2.1pt);
					\node[right=3pt] at (0,-1.35) {infeasible; not analysed};
				\end{scope}
			\end{tikzpicture}}
		\caption{Placement cube $\{0,\tau\}^3$ and outcomes}
		\label{fig:schematic-cube}
	\end{subfigure}
	\caption{Common-delay placements in \eqref{eq:family}.
		(a) Interactions among pollution $P$, visitors $V$, and capital $K$; coloured badges mark where each delay enters.
		(b) The eight placements shown as corners of $\{0,\tau\}^3$.
		On the forward-invariant shaded face $\tau_{1}=0$, three placements are delay-independently stable, while $(0,\tau,0)$ undergoes a Hopf bifurcation.
		The hatched face $\tau_{1}=\tau$ is not forward invariant; only $(\tau,0,\tau)$ is analysed locally.}
	\label{fig:schematic}
\end{figure}
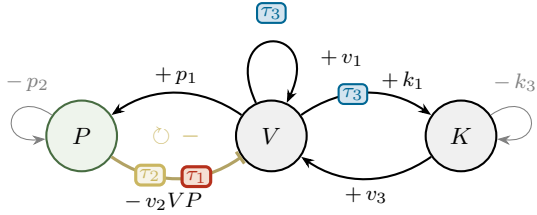
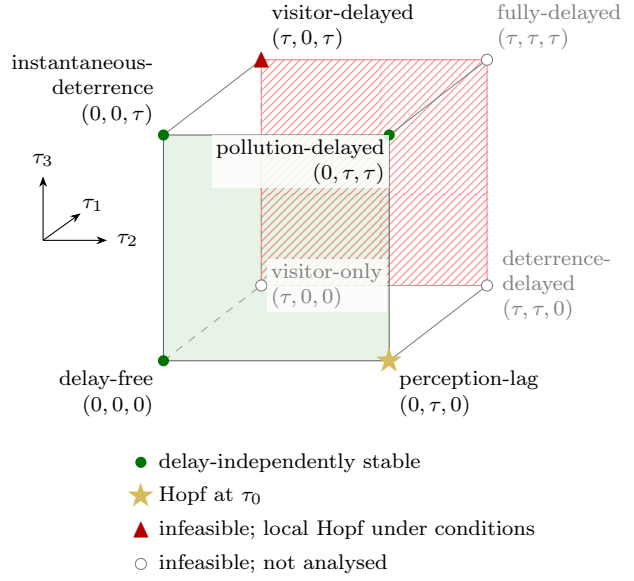

\begin{table}[t]
	\centering
	\footnotesize
	\setlength{\tabcolsep}{3pt}
	\begin{tabular}{@{}llc>{\raggedright\arraybackslash}p{4.9cm}@{}}
		\toprule
		$(\tau_{1},\tau_{2},\tau_{3})$ & Placement                & Feasible?   & Result                                                           \\
		\midrule
		$(0,0,0)$                      & delay-free               & yes         & locally asymptotically stable equilibrium (Prop.~\ref{prop:nondelayed-stability})   \\
		$(0,0,\tau)$                   & instantaneous-deterrence & yes         & delay-independently stable (Thm.~\ref{thm:stable-feasible})      \\
		$(0,\tau,\tau)$                & pollution-delayed        & yes         & delay-independently stable (Thm.~\ref{thm:stable-feasible})         \\
		$(0,\tau,0)$                   & perception-lag           & yes         & Hopf bifurcation for every set of positive parameters (Thm.~\ref{thm:perception-hopf})        \\
		\midrule
		$(\tau,0,\tau)$                & visitor-delayed          & no          & Hopf bifurcation under transversality and nonresonance conditions (Prop.~\ref{prop:visitor-hopf})              \\
		$(\tau,\tau,\tau)$             & fully-delayed            & no          & not analysed                                                    \\
		$(\tau,0,0)$                   & visitor-only             & no          & not analysed                                                    \\
		$(\tau,\tau,0)$                & deterrence-delayed       & no          & not analysed                                                    \\
		\bottomrule
	\end{tabular}
	\caption{Eight common-delay placements for \eqref{eq:family}; slots are ordered as the visitor factor in deterrence, pollution perception, and visitation-driven growth and reinvestment.
		Stability is local and asymptotic; the visitor-delayed Hopf result requires a simple, nonresonant, transversal crossing.}
	\label{tab:corners}
\end{table}

Our main results are summarised in Table~\ref{tab:corners}.
The four placements on the face $\tau_{1}=0$ are precisely those for which every nonnegative history generates a global nonnegative solution.
Three have a delay-independently stable positive equilibrium; the perception-lag placement $(0,\tau,0)$ instead loses stability at a finite delay for every positive parameter set, and its Hopf bifurcation can be supercritical or subcritical.
The visitor-delayed comparison shows that a local Hopf bifurcation can also occur in a placement that is not feasible.

The paper is organised as follows.
Section~\ref{sec:analysis} proves the feasibility dichotomy, establishes global existence for the four feasible placements, develops the Hopf framework, applies it to the feasible placements, and gives a local treatment of one infeasible placement.
Section~\ref{sec:numerics} presents the supercritical and subcritical scenarios, numerical continuation, and the parameter dependence of the bifurcation direction.
Appendix~\ref{app:cm} contains the center-manifold reduction, and Appendix~\ref{app:caraballo} analyses the positivity-preserving modification of Caraballo et al.

\FloatBarrier
\section{Model and Analysis}\label{sec:analysis}

Throughout this section, all model parameters are positive and $\tau\ge0$.

\subsection{Common Equilibrium and Delay-Free Stability}\label{ssec:nondelayed}

At $\tau=0$, the models corresponding to all eight placements coincide, so the equilibrium and delay-free analysis are common to the whole family.
Write the common system as $du/dt=f(u)$, with $u=(V,P,K)$ and $f=(f_{V},f_{P},f_{K})$.
This system is quasi-positive \cite{efendiev2013}.
That is, its vector field points inward or tangentially on each boundary face of the nonnegative orthant:
\[
	f_V(0,P,K)=v_3K\ge0,\qquad f_P(V,0,K)=p_1V\ge0,\qquad f_K(V,P,0)=k_1V\ge0.
\]
Thus solutions with nonnegative initial data remain nonnegative.

\begin{proposition}\label{prop:positive-equilibrium}
	Every placement in the family \eqref{eq:family} has the same unique positive equilibrium $F$, given by
	\begin{equation}\label{eq:equilibrium}
		F=\left(V_{\infty}, P_{\infty}, K_{\infty}\right),
	\end{equation}
	where
	\begin{equation}\label{eq:equilibrium-formulas}
		V_{\infty}=\frac{v_1 p_2 k_3+v_3 p_2 k_1}{v_2 p_1 k_3}, \qquad
		P_{\infty}=\frac{p_1}{p_2} V_{\infty}, \qquad
		K_{\infty}=\frac{k_1}{k_3} V_{\infty}.
	\end{equation}
\end{proposition}

\begin{proof}
	At an equilibrium, all delayed and current state values coincide.
	From $\dot P=\dot K=0$,
	\[
		P_\infty=\frac{p_1}{p_2}V_\infty,\qquad K_\infty=\frac{k_1}{k_3}V_\infty.
	\]
	Substituting into $\dot V=0$ gives
	\[
		V_\infty\left(v_1+\frac{v_3k_1}{k_3}-\frac{v_2p_1}{p_2}V_\infty\right)=0,
	\]
	whose unique positive solution is the value of $V_\infty$ in \eqref{eq:equilibrium-formulas}.
\end{proof}

\begin{remark}\label{rem:active}
	From \eqref{eq:equilibrium-formulas}, the equilibrium pollution level satisfies the identity
	\[
		v_{2}P_{\infty}-v_{1}=\frac{v_{3}k_{1}}{k_{3}}>0,
	\]
	which is used repeatedly in the stability and center-manifold analysis below.
\end{remark}

\begin{proposition}\label{prop:nondelayed-stability}
	For the delay-free system, the equilibrium $F$ is locally asymptotically stable.
\end{proposition}
\begin{proof}
	At $F=\left(V_{\infty}, P_{\infty}, K_{\infty}\right)$, the Jacobian is
	\[
		J(F)=
		\begin{pmatrix}
			v_1-v_2 P_{\infty} & -v_2 V_{\infty} & v_3  \\
			p_1                & -p_2            & 0    \\
			k_1                & 0               & -k_3
		\end{pmatrix}.
	\]

	Expanding $\det(\lambda I-J(F))=\lambda^{3}+A_{1}\lambda^{2}+A_{2}\lambda+A_{3}$, using \eqref{eq:equilibrium-formulas} and the identity $v_1-v_2P_{\infty}=-v_3k_1/k_3$ of Remark~\ref{rem:active}, gives
	\[
		A_1=p_2+k_3+\frac{v_3k_1}{k_3},\qquad
		A_2=p_2\left(v_1+k_3+\frac{2v_3k_1}{k_3}\right),\qquad
		A_3=p_2(v_1k_3+v_3k_1),
	\]
	all positive, and
	\[
		A_1A_2-A_3 =
		p_2\left[
		p_2\left(v_1+k_3+\frac{2v_3k_1}{k_3}\right)
		+k_3^2
		+\frac{v_3k_1}{k_3}
		\left(v_1+2k_3+\frac{2v_3k_1}{k_3}\right)
		\right]>0.
	\]
	By the Routh--Hurwitz criterion, all eigenvalues of $J(F)$ have negative real parts and $F$ is locally asymptotically stable.
\end{proof}

\subsection{Feasibility and Global Existence}\label{ssec:noninv}

The phase space is the Banach space $C([-\tau,0],\mathbb{R}^{3})$ of history segments $x_{t}(\theta):=x(t+\theta)$, $\theta\in[-\tau,0]$, under the supremum norm, and the solution operators $\Phi(t):\phi\mapsto x_{t}$ form a semiflow.
Because each model obtained from \eqref{eq:family} has a $C^{1}$ right-hand side and involves at most the common discrete delay $\tau$, a history $\phi=(\phi_{V},\phi_{P},\phi_{K})$ determines a unique maximal solution on some interval $[0,T)$, depending continuously on $\phi$ \cite{hale1993}; the solution is \emph{global} when $T=\infty$.
The cone of nonnegative histories $C([-\tau,0],[0,\infty)^{3})$ is \emph{forward invariant} when $\Phi(t)$ maps it into itself, that is, when every nonnegative history yields a solution that stays nonnegative for as long as it exists.
We call a placement \emph{feasible} if every nonnegative history generates a global solution that remains nonnegative.
Forward invariance supplies the nonnegativity part of this definition; Lemma~\ref{lem:global-existence} below shows that it also entails global existence for the present family.
By Seifert's quasi-positivity criterion \cite{seifert1976,hal2011}, forward invariance holds if and only if $f_i(\phi)\ge0$ for every nonnegative history $\phi$ with $\phi_i(0)=0$.

\begin{lemma}[Feasibility dichotomy]\label{lem:cone-general}
	For every $\tau>0$, a placement of \eqref{eq:family} leaves the cone of nonnegative histories $C([-\tau,0],[0,\infty)^3)$ forward invariant if and only if $\tau_{1}=0$.
\end{lemma}
\begin{proof}
	On the faces $\{P=0\}$ and $\{K=0\}$, the quasi-positivity inequalities hold unconditionally:
	\[
		f_{P}(\phi)\big|_{\phi_{P}(0)=0}=p_{1}\phi_{V}(0)\ge0,\qquad
		f_{K}(\phi)\big|_{\phi_{K}(0)=0}=k_{1}\phi_{V}(-\tau_{3})\ge0.
	\]
	The latter holds for either value of $\tau_{3}$.
	Thus only the face $\{V=0\}$ can obstruct invariance; there,
	\[
		f_{V}(\phi)\big|_{\phi_{V}(0)=0}=v_{1}\phi_{V}(-\tau_{3})-v_{2}\phi_{V}(-\tau_{1})\phi_{P}(-\tau_{2})+v_{3}\phi_{K}(0).
	\]
	If $\tau_{1}=0$, then $\phi_{V}(-\tau_{1})=0$, so the cross term vanishes and $f_{V}\big|_{\phi_{V}(0)=0}\ge0$.
	The cone is therefore forward invariant.
	If $\tau_{1}=\tau>0$, take any nonnegative history $\phi$ with
	\begin{equation}\label{eq:boundary-history}
		\phi_{V}(0)=\phi_{K}(0)=0, \qquad \phi_{V}(-\tau)>0, \qquad \phi_{P}(-\tau_{2})>v_{1}/v_{2}.
	\end{equation}
	Since $\phi_{V}(-\tau_{3})\in\{\phi_{V}(0),\phi_{V}(-\tau)\}$,
	\begin{equation}\label{eq:v_factored}
		\dot V(0)=v_{1}\phi_{V}(-\tau_{3})-v_{2}\phi_{V}(-\tau)\phi_{P}(-\tau_{2})
		\le -\phi_{V}(-\tau)\bigl(v_{2}\phi_{P}(-\tau_{2})-v_{1}\bigr)<0,
	\end{equation}
	so $V(t)<0$ for small $t>0$ and the cone is not forward invariant.
\end{proof}

By \eqref{eq:v_factored} the solution from the history \eqref{eq:boundary-history} satisfies $V(t_*)<0$ at a fixed time $t_*>0$; by continuous dependence, so does every solution from a nearby history.
Such histories include ones with strictly positive $V$ and $K$, which therefore cross the face $\{V=0\}$ in finite time.

\begin{lemma}\label{lem:global-existence}
	If the solution of a placement of \eqref{eq:family} with a nonnegative history remains nonnegative on its maximal interval $[0,T)$, then $T=\infty$.
	Moreover, its components are bounded on every finite interval.
\end{lemma}

\begin{proof}
	While the solution is nonnegative the cross term $-v_2V(t-\tau_{1})P(t-\tau_{2})\le0$, so
	\[
		\dot V(t)\le v_1V(t-\tau_{3})+v_3K(t),\qquad \dot K(t)\le k_1V(t-\tau_{3}).
	\]
	Set $S=V+K$ and $c=v_1+k_1+v_3$.
	Since $t-\tau_{3}\in\{t,t-\tau\}$, both $V(t-\tau_{3})$ and $K(t)$ are bounded by $\sup_{t-\tau\le s\le t}S(s)$.
	Adding the two inequalities gives $\dot S(t)\le c\sup_{t-\tau\le s\le t}S(s)$.
	Write $m(t):=\sup_{-\tau\le s\le t}S(s)$ and note $m(t)=\max\bigl(m(0),\,\sup_{0\le u\le t}S(u)\bigr)$.
	Integrating over $[0,u]$ gives $S(u)\le m(0)+c\int_0^t m(r)\,dr$ for every $u\in[0,t]$.
	Taking the supremum over $u\in[0,t]$ bounds the second entry of the maximum by the right-hand side, which also dominates the first entry $m(0)$.
	Hence $m(t)\le m(0)+c\int_0^t m(r)\,dr$, and Gronwall's inequality yields $m(t)\le m(0)e^{ct}$.
	Thus $V$ and $K$ are bounded on every finite interval, and using variation of constants in $\dot P=p_1V-p_2P$ then bounds $P$.
	By the continuation criterion for retarded functional differential equations \cite[\S2.3]{hale1993}, $T=\infty$.
\end{proof}

\begin{corollary}\label{cor:global-feasible}
	Every nonnegative history for a feasible placement generates a unique global nonnegative solution.
\end{corollary}
\begin{proof}
	Combine forward invariance from Lemma~\ref{lem:cone-general} with Lemma~\ref{lem:global-existence}.
\end{proof}

\begin{remark}\label{rem:he-noninv}
	The two-delay tourism models \cite{he2022,he2023} have visitor equation
	\[
		\dot V=\alpha_{1}V(t-\tau_{1})-\alpha_{2}V^{2}(t-\tau_{1})+\cdots,
	\]
	whose omitted terms carry current-time stock factors.
	For a nonnegative history with all current stocks zero, these terms vanish.
	If $\phi_{V}(-\tau_{1})>\alpha_{1}/\alpha_{2}$, then
	\[
		f_{V}(\phi)=\phi_{V}(-\tau_{1})\bigl(\alpha_{1}-\alpha_{2}\phi_{V}(-\tau_{1})\bigr)<0.
	\]
	Thus, by the same criterion, the cone is not forward invariant for $\tau_{1}>0$.
	The corrective factor of \cite{caraballo2019}, already adopted in \cite{he2023}, restores invariance only on the environmental boundary and leaves this obstruction intact.
\end{remark}

\subsection{The Characteristic Equation}\label{ssec:framework}

Every placement in the placement cube shares the equilibrium $F$.
A Hopf bifurcation can occur only when a pair of characteristic roots reaches the imaginary axis.
We therefore convert the search for imaginary roots into a polynomial root problem and then recover the delay values at which those roots occur.
Linearising any placement at $F$ yields a characteristic equation of the form
\begin{equation}\label{eq:quasipolynomial}
	\begin{aligned}
		\mathcal D(\lambda,\tau)&:=\mathcal P(\lambda)+\mathcal Q(\lambda)e^{-\lambda\tau}=0, \\
		\mathcal P(\lambda)&=\lambda^{3}+a_{0}\lambda^{2}+a_{1}\lambda+a_{2}, \\
		\mathcal Q(\lambda)&=b_{0}\lambda^{2}+b_{1}\lambda+b_{2}.
	\end{aligned}
\end{equation}
Here $\mathcal D$ is the \emph{characteristic function}, and the placement determines the coefficients $a_i,b_i$.
Solutions of \eqref{eq:quasipolynomial} are called \emph{characteristic roots}.
A nonzero purely imaginary characteristic root $\lambda=i\omega$, with $\omega>0$, requires $\mathcal P(i\omega)=-\mathcal Q(i\omega)e^{-i\omega\tau}$.
Taking squared moduli eliminates $\tau$ and, in the variable $\xi=\omega^{2}$, yields the auxiliary cubic
\begin{equation}\label{eq:cubic-in-z}
	\mathcal H(\xi) = \xi^3 + p \xi^2 + q \xi + n = 0,
\end{equation}
where
\begin{equation}\label{eq:cubic-coefficients}
	p =a_0^2-2 a_1 - b_0^2, \qquad
	q =a_1^2 - 2 a_0 a_2 -b_1^2 + 2 b_0 b_2, \qquad
	n =a_2^2-b_2^2.
\end{equation}
Since $\xi=\omega^{2}$ with $\omega>0$, only the positive roots of $\mathcal H$ are relevant in \eqref{eq:cubic-in-z}; we write $\Xi$ for their number, counted with multiplicity.

The characteristic coefficients of the analysed placements are summarised in Table~\ref{tab:coefficient-ledger}, using the abbreviations
\begin{equation}\label{eq:M-d}
	M:=v_{2}P_{\infty},\qquad d:=M-v_{1}=\frac{v_{3}k_{1}}{k_{3}}>0,
\end{equation}
the identity and sign coming from Remark~\ref{rem:active}.

\begin{table}[htbp]
	\centering
	\footnotesize
	\setlength{\tabcolsep}{4pt}
	\begin{tabular}{@{}lccccccc@{}}
		\toprule
		Placement & $a_0$ & $a_1$ & $a_2$ & $b_0$ & $b_1$ & $b_2$ & $\operatorname{sign}n$ \\
		\midrule
		$(0,0,\tau)$ & $p_2+k_3+M$ & $p_2k_3+2p_2M+k_3M$ & $2p_2k_3M$ & $-v_1$ & $-p_2v_1-k_3M$ & $-p_2k_3M$ & $+$ \\
		$(0,\tau,\tau)$ & $p_2+k_3+M$ & $p_2k_3+p_2M+k_3M$ & $p_2k_3M$ & $-v_1$ & $p_2d-k_3M$ & $0$ & $+$ \\
		$(0,\tau,0)$ & $p_2+d+k_3$ & $p_2d+p_2k_3$ & $0$ & $0$ & $p_2M$ & $p_2k_3M$ & $-$ \\
		$(\tau,0,\tau)$ & $p_2+k_3$ & $p_2M+p_2k_3$ & $p_2k_3M$ & $d$ & $p_2d$ & $0$ & $+$ \\
		\bottomrule
	\end{tabular}
	\caption{Characteristic coefficients for the four analysed placements in \eqref{eq:quasipolynomial}, where $M=v_{2}P_{\infty}$ and $d=M-v_{1}=v_{3}k_{1}/k_{3}>0$; the final column gives the sign of $n=a_{2}^{2}-b_{2}^{2}$.}
	\label{tab:coefficient-ledger}
\end{table}

Since $\mathcal H(0)=n\ne0$ for the rows in Table~\ref{tab:coefficient-ledger} and $\mathcal H(\xi)\to+\infty$ as $\xi\to+\infty$, the sign of $n$ controls the parity of $\Xi$, the number of positive roots counted with multiplicity:
\begin{equation}\label{eq:n-sign}
	n<0 \;\Longrightarrow\; \text{$\Xi$ odd},
	\qquad
	n>0 \;\Longrightarrow\; \text{$\Xi$ even}.
\end{equation}
In particular, $n<0$ forces a positive root, while $n>0$ is inconclusive, as the even count includes zero.

\begin{lemma}[Root-counting principle]\label{lem:root-counting}
	If, for a placement, $\Xi=0$ and $a_{2}+b_{2}\neq0$, then \eqref{eq:quasipolynomial} has no characteristic root on the imaginary axis for any $\tau\ge0$, and $F$ is locally asymptotically stable for every finite $\tau\ge0$, with no Hopf bifurcation.
\end{lemma}
\begin{proof}
	Since $\Xi=0$ and at $\lambda=0$ the left side of \eqref{eq:quasipolynomial} equals $a_{2}+b_{2}\neq0$ for every $\tau$, no characteristic root lies on the imaginary axis.
	Characteristic roots depend continuously on $\tau$; at each delay only finitely many lie in any fixed right half-plane, and on a compact delay interval these remain uniformly bounded \cite{hale1993,ruan2003}.
	Consequently, the number with $\operatorname{Re}\lambda>0$ can change only when a characteristic root crosses the imaginary axis; roots cannot enter the right half-plane from infinity at a finite delay.
	As the axis is root-free, this number is independent of $\tau$.
	At $\tau=0$ it is zero (Proposition~\ref{prop:nondelayed-stability}), so all characteristic roots have negative real part and $F$ is locally asymptotically stable for every finite $\tau\ge0$.
\end{proof}

\subsection{The Critical Delay and Transversality}\label{ssec:crossing}

The results of this subsection are classical \cite{kuang1993,beretta2002,ruan2003}; we prove them once for the common characteristic form \eqref{eq:quasipolynomial}, so that each analysed placement is handled by substituting its coefficient row from Table~\ref{tab:coefficient-ledger}.

\begin{proposition}\label{prop:imaginary-roots}
	Assume that $\mathcal H$ has a positive root $\xi_{0}$ and that $\mathcal P$ and $\mathcal Q$ have no common zero on the imaginary axis.
	Set $\omega_{0}=\sqrt{\xi_{0}}$ and
	\begin{equation}\label{eq:CS}
		\begin{aligned}
			\mathcal{C}
			&=\frac{(b_{1}-a_{0}b_{0})\omega_{0}^{4}+(a_{0}b_{2}+a_{2}b_{0}-a_{1}b_{1})\omega_{0}^{2}-a_{2}b_{2}}
			{(b_{2}-b_{0}\omega_{0}^{2})^{2}+b_{1}^{2}\omega_{0}^{2}}, \\
			\mathcal{S}
			&=\frac{(a_{1}\omega_{0}-\omega_{0}^{3})(b_{2}-b_{0}\omega_{0}^{2})-(a_{2}-a_{0}\omega_{0}^{2})b_{1}\omega_{0}}
			{(b_{2}-b_{0}\omega_{0}^{2})^{2}+b_{1}^{2}\omega_{0}^{2}}.
		\end{aligned}
	\end{equation}
	Then $\lambda=\pm i\omega_{0}$ are characteristic roots of \eqref{eq:quasipolynomial} exactly at the delays
	\begin{equation}\label{tau0form}
		\tau_{0}^{(k)}=\frac{1}{\omega_{0}}\left\{\Theta(\omega_{0})+2k\pi\right\},\qquad
		\Theta(\omega_{0})=\begin{cases}
			\cos^{-1}(\mathcal{C}),      & \mathcal{S}\ge 0, \\[2pt]
			2\pi-\cos^{-1}(\mathcal{C}), & \mathcal{S}<0,
		\end{cases}
		\qquad k\in\mathbb{N}_{0}.
	\end{equation}
	All these delays are positive under the delay-free stability established in Proposition~\ref{prop:nondelayed-stability}.
\end{proposition}
\begin{proof}
	The common denominator in \eqref{eq:CS} is $(b_{2}-b_{0}\omega_{0}^{2})^{2}+b_{1}^{2}\omega_{0}^{2}=|\mathcal Q(i\omega_{0})|^{2}$.
	It is strictly positive: $\mathcal H(\xi_{0})=0$ is exactly the modulus identity $|\mathcal P(i\omega_{0})|^{2}=|\mathcal Q(i\omega_{0})|^{2}$, so a zero denominator would make $i\omega_{0}$ a common zero of $\mathcal P$ and $\mathcal Q$.
	The real and imaginary parts of \eqref{eq:quasipolynomial} at $\lambda=i\omega_{0}$ can therefore be solved for $\cos(\omega_{0}\tau)$ and $\sin(\omega_{0}\tau)$.
	This gives $\cos(\omega_{0}\tau)=\mathcal{C}$ and $\sin(\omega_{0}\tau)=\mathcal{S}$.
	A real $\tau$ satisfies both equations if and only if $\mathcal{C}^{2}+\mathcal{S}^{2}=1$, which is again the identity $\mathcal H(\xi_{0})=0$.
	The solutions are precisely the sequence \eqref{tau0form}; the case split in $\Theta$ is required because the principal value $\cos^{-1}(\mathcal{C})\in[0,\pi]$ fixes the angle only when $\mathcal{S}\ge0$.
	Finally, $\Theta(\omega_{0})>0$.
	Otherwise $\mathcal{C}=1$ and $\tau_{0}^{(0)}=0$, making $\lambda=i\omega_{0}$ a root of the delay-free characteristic polynomial, all of whose roots have negative real part by Proposition~\ref{prop:nondelayed-stability}.
	Hence every $\tau_{0}^{(k)}$ is positive.
\end{proof}

We write $\tau_{0}^{(k)}(\xi)$ for the delays \eqref{tau0form} associated with a positive root $\xi$ of $\mathcal H$; when $\Xi\ge1$, the \emph{critical delay} is $\tau_{0}=\min_{\xi}\tau_{0}^{(0)}(\xi)$, possibly attained by more than one root.
The pair $(F,\tau_{0})$ is called a \emph{Hopf point} when the critical roots have multiplicity one (simplicity), cross the imaginary axis with nonzero speed (transversality), and are the only roots on that axis at the same delay (nonresonance).
The resulting emergence of periodic solutions is the \emph{Hopf bifurcation}.

\begin{lemma}[Transversality]\label{lem:transversality}
	Let $\xi_{0}$ be a positive root of $\mathcal H$, assume that $\mathcal P$ and $\mathcal Q$ have no common zero on the imaginary axis, and let $\tau_{0}^{(k)}(\xi_{0})$ be any of the associated delays \eqref{tau0form}.
	Set $\omega_{0}=\sqrt{\xi_{0}}$.
	If $\mathcal H^{\prime}(\xi_{0})\ne0$, then the characteristic roots $\pm i\omega_{0}$ are simple, and the crossing speed
	\[
		\frac{d}{d\tau}\operatorname{Re}(\lambda(\tau))\Big|_{\tau=\tau_{0}^{(k)}(\xi_{0})}
	\]
	has the sign of $\mathcal H^{\prime}(\xi_{0})$; in particular it is nonzero, so the crossing is transversal.
\end{lemma}

\begin{proof}
	Write $\mathcal D_\lambda$ and $\mathcal D_\tau$ for the partial derivatives of the characteristic function.
	At $\lambda=i\omega_{0}$, one has $\mathcal D_\tau=-\lambda\mathcal Q(\lambda)e^{-\lambda\tau}=\lambda\mathcal P(\lambda)\ne0$ because $\mathcal Q(i\omega_{0})\ne0$.
	Using $\mathcal Q(\lambda)e^{-\lambda\tau}=-\mathcal P(\lambda)$ to eliminate the exponential gives
	\begin{equation}\label{eq:dldtau-inv}
		-\frac{\mathcal D_\lambda}{\mathcal D_\tau}
		=-\,\frac{\mathcal P^{\prime}(\lambda)+\mathcal Q^{\prime}(\lambda)e^{-\lambda\tau}}{\lambda\,\mathcal P(\lambda)}-\frac{\tau}{\lambda}.
	\end{equation}
	The last term is purely imaginary at $\lambda=i\omega_{0}$.
	Evaluating the remaining expression and using $\omega_{0}^{2}=\xi_{0}$ and $\mathcal H(\xi_{0})=0$ gives
	\[
		\operatorname{Re}\left(-\frac{\mathcal D_\lambda}{\mathcal D_\tau}\right)_{\lambda=i\omega_{0}}
		=\frac{\mathcal H^{\prime}(\xi_{0})}{|\mathcal Q(i\omega_{0})|^{2}},
	\]
	whose right-hand side is nonzero by hypothesis.
	Hence $\mathcal D_\lambda\ne0$, so the characteristic roots are simple and the implicit branch $\lambda(\tau)$ exists.
	Since $(d\lambda/d\tau)^{-1}=-\mathcal D_\lambda/\mathcal D_\tau$ and $\operatorname{Re}z$ and $\operatorname{Re}(1/z)$ have the same sign for $z\ne0$, the derivative has the sign of $\mathcal H^{\prime}(\xi_{0})$ and is nonzero.
\end{proof}

\begin{theorem}\label{thm:delay-hopf}
	Assume that $\mathcal P$ and $\mathcal Q$ have no common zero on the imaginary axis and that $a_{2}+b_{2}\neq0$.
	\begin{enumerate}
		\item If $\Xi\ge1$, let $\tau_{0}$ be the critical delay, attained at a positive root $\xi_{0}$ of $\mathcal H$, and set $\omega_{0}=\sqrt{\xi_{0}}$.
		      If $\mathcal H^{\prime}(\xi_{0})\ne0$ (simplicity and transversality) and $\pm i\omega_{0}$ are the only characteristic roots of \eqref{eq:quasipolynomial} on the imaginary axis at $\tau=\tau_{0}$ (nonresonance), then $F$ is locally asymptotically stable for $\tau\in[0,\tau_{0})$ and undergoes a Hopf bifurcation at $\tau=\tau_{0}$.
		\item If $\Xi=0$, then $F$ is locally asymptotically stable for every $\tau\ge0$, with no Hopf bifurcation.
	\end{enumerate}
\end{theorem}

\begin{proof}
	Case (2) is Lemma~\ref{lem:root-counting}.
	For case (1), the sequence $\tau_{0}^{(k)}(\xi)=\tau_{0}^{(0)}(\xi)+2k\pi/\sqrt{\xi}$ increases in $k$, so the critical delay $\tau_{0}$ is the smallest delay at which \eqref{eq:quasipolynomial} has a purely imaginary characteristic root.
	As in the proof of Lemma~\ref{lem:root-counting}, the number of characteristic roots with $\operatorname{Re}\lambda>0$ is zero at $\tau=0$ and can change only at such a crossing, so $F$ remains locally asymptotically stable on $[0,\tau_{0})$.
	At $\tau=\tau_{0}$ the simple pair $\pm i\omega_{0}$ crosses the imaginary axis transversally (Lemma~\ref{lem:transversality}), and the nonresonance hypothesis then yields a Hopf bifurcation \cite{hale1993,hassard1981}.
\end{proof}

\subsection{Direction and Stability of the Hopf Bifurcation}\label{ssec:direction}

At a Hopf point $(F,\tau_{0})$, the coefficient $c_1(0)$ determines the direction and orbital stability of the bifurcating branch.
Let $\lambda(\tau)$ be the critical characteristic root in original time.
After rescaling time so that the delay is one, the corresponding characteristic root is $\Lambda(\tau)=\tau\lambda(\tau)$, with $\Lambda(\tau_{0})=i\omega_{0}\tau_{0}$.
At $\tau=\tau_{0}$, the center-manifold equation is
\begin{equation}
	\label{eq:reduced}
	\frac{dz}{dt} = i\omega_{0}\tau_{0} z + g(z,\bar{z}),
\end{equation}
where
\[
	g(z, \bar{z}) = g_{20} \frac{z^2}{2} + g_{11} z \bar{z} + g_{02} \frac{\bar{z}^2}{2} + g_{21} \frac{z^2 \bar{z}}{2} + \cdots,
\]
with $g_{jk}$ denoting the normal-form coefficients at the Hopf point.
The coefficient $c_1(0)$ is given by
\begin{equation}\label{mukx}
	c_1(0)=\frac{i}{2 \omega_{0} \tau_{0}}\left(g_{20} g_{11}-2\left|g_{11}\right|^2-\frac{1}{3}\left|g_{02}\right|^2\right)+\frac{1}{2}g_{21}.
\end{equation}
Because $\tau_{0}$ is the first crossing and $F$ is stable for $\tau<\tau_{0}$, transversality forces the critical pair rightward, so $\operatorname{Re}\lambda'(\tau_{0})>0$.
Since $\operatorname{Re}\lambda(\tau_{0})=0$, it follows that $\operatorname{Re}\Lambda'(\tau_{0})=\tau_{0}\operatorname{Re}\lambda'(\tau_{0})>0$.
The unfolding coefficient
\[
	\mu_2=-\frac{\operatorname{Re}c_1(0)}{\operatorname{Re}\Lambda'(\tau_{0})}
\]
therefore has the sign of $-\operatorname{Re}c_1(0)$.
Thus, when $\operatorname{Re}c_1(0)\ne0$, the sign of $\operatorname{Re}c_1(0)$ determines both the direction and the orbital stability of the branch.

\begin{theorem}\label{thm:hopf}
	Let $(F,\tau_{0})$ be a Hopf point with $\operatorname{Re}c_1(0)\ne0$.
	If $\operatorname{Re}c_1(0)<0$, the Hopf bifurcation is supercritical: nearby periodic solutions exist for $\tau>\tau_{0}$ sufficiently close to $\tau_{0}$ and are orbitally asymptotically stable, meaning that nearby trajectories approach the periodic orbit up to a phase shift.
	If $\operatorname{Re}c_1(0)>0$, it is subcritical: nearby periodic solutions exist for $\tau<\tau_{0}$ sufficiently close to $\tau_{0}$ and are orbitally unstable, meaning that perturbations can move trajectories away from the periodic orbit.
\end{theorem}

Appendix~\ref{app:cm} derives the normal-form coefficients for the perception-lag model.
Appendix~\ref{app:cm-visitor} records the matrix and quadratic-term replacements for the visitor-delayed model.

\subsection{Delay-Independent Stability of Two Feasible Placements}\label{ssec:stable-feasible}

Here \emph{delay-independent stability} means that $F$ is locally asymptotically stable for every $\tau\ge0$ and that no Hopf bifurcation occurs.

\begin{theorem}\label{thm:stable-feasible}
	For the pollution-delayed model $(0,\tau,\tau)$ and the instantaneous-deterrence model $(0,0,\tau)$, the positive equilibrium $F$ is delay-independently stable.
\end{theorem}
\begin{proof}
	For both placements $n>0$, so \eqref{eq:n-sign} is inconclusive; we therefore verify the hypotheses $\Xi=0$ and $a_{2}+b_{2}\neq0$ of Lemma~\ref{lem:root-counting} directly.

	\emph{(a) Pollution-delayed.}
	Substituting the pollution-delayed row of Table~\ref{tab:coefficient-ledger} into \eqref{eq:cubic-coefficients}, using $b_{2}=0$ and $M^{2}-d^{2}=v_{1}^{2}+2v_{1}d$, gives
	\[
		n=a_{2}^{2}>0,\qquad
		p=p_{2}^{2}+k_{3}^{2}+M^{2}-v_{1}^{2}>0,\qquad
		q=p_{2}\Bigl[p_{2}\bigl(k_{3}^{2}+v_{1}^{2}+2v_{1}d\bigr)+2Mk_{3}d\Bigr]>0,
	\]
	where $p>0$ because $M>v_{1}$.
	Every coefficient of $\mathcal H$ is thus positive, so $\Xi=0$; moreover $a_{2}+b_{2}=a_{2}>0$.

	\emph{(b) Instantaneous-deterrence.}
	For the instantaneous-deterrence row of Table~\ref{tab:coefficient-ledger}, the coefficients $p,q$ of the auxiliary cubic are not sign-definite in dimensional form, so we introduce the dimensionless ratios
	\[
		a=\frac{p_{2}}{M},\qquad b=\frac{k_{3}}{M},\qquad r=\frac{v_{1}}{M}\in(0,1),\qquad \xi=M^{2}y,
	\]
	in which
	\[
		\begin{aligned}
			\frac{\mathcal H(M^{2}y)}{M^{6}}
			&=y^{3}+\bigl(a^{2}-2a+b^{2}+1-r^{2}\bigr)y^{2} \\
			&\quad+a\bigl[a(b^{2}+4-r^{2})-2b^{2}\bigr]y+3a^{2}b^{2} \\
			&=H(y)+(1-r^{2})\bigl(y^{2}+a^{2}y\bigr),
		\end{aligned}
	\]
	where
	\[
		H(y):=(b^{2}+y)\bigl(y^{2}+(a^{2}-2a)y+3a^{2}\bigr).
	\]
	For $y>0$, both factors of $H$ are positive, since
	\[
		y^{2}+(a^{2}-2a)y+3a^{2}=(y-a)^{2}+2a^{2}+a^{2}y>0.
	\]
	Because also $1-r^{2}>0$ and $y^{2}+a^{2}y>0$, we have $\mathcal H(\xi)>0$ for every $\xi>0$, hence $\Xi=0$; moreover $a_{2}+b_{2}=p_{2}k_{3}M>0$.

	In both cases Lemma~\ref{lem:root-counting} now gives the result.
\end{proof}

\subsection{The Perception-Lag Model}\label{ssec:perception}

The only feasible placement that admits delay-induced destabilisation is the perception-lag model
$(\tau_{1},\tau_{2},\tau_{3})=(0,\tau,0)$:
\begin{equation}\label{eq:perception}
	\begin{aligned}
		 & \dot V=v_{1}V(t)-v_{2}V(t)P(t-\tau)+v_{3}K(t), \\
		 & \dot P=p_{1}V(t)-p_{2}P(t),                    \\
		 & \dot K=k_{1}V(t)-k_{3}K(t).
	\end{aligned}
\end{equation}
By Corollary~\ref{cor:global-feasible}, every nonnegative history generates a unique global nonnegative solution.
\begin{theorem}\label{thm:perception-hopf}
	For every choice of positive model parameters:
	\begin{enumerate}
		\item The auxiliary cubic has a unique positive root $\xi_{0}$, which is simple and satisfies $\mathcal H^{\prime}(\xi_{0})>0$.
		      The critical delay is $\tau_{0}=\tau_{0}^{(0)}(\xi_{0})$, with frequency $\omega_{0}=\sqrt{\xi_{0}}$.
		\item At $\tau=\tau_{0}$, the characteristic equation has the purely imaginary roots $\pm i\omega_{0}$, while $F$ is locally asymptotically stable for $\tau\in[0,\tau_{0})$.
		\item This is the only characteristic-root pair on the imaginary axis at $\tau=\tau_{0}$, and it crosses transversally into the right half-plane.
		      Hence $(F,\tau_{0})$ is a Hopf point with a simple critical pair.
		\item The equilibrium $F$ is unstable for every $\tau>\tau_{0}$; in particular, it does not restabilise as the delay increases.
	\end{enumerate}
\end{theorem}
\begin{proof}
	Linearising \eqref{eq:perception} at $F$ and using the identities $p_{1}v_{2}V_{\infty}=p_{2}M$ and $v_{3}k_{1}=d\,k_{3}$ from \eqref{eq:equilibrium-formulas} and \eqref{eq:M-d}, the characteristic equation factors as
	\[
		\begin{aligned}
			\mathcal P(\lambda)+\mathcal Q(\lambda)e^{-\lambda\tau}&=0, \\
			\mathcal P(\lambda)&=\lambda(\lambda+p_{2})(\lambda+d+k_{3}), \\
			\mathcal Q(\lambda)&=p_{2}M(\lambda+k_{3}),
		\end{aligned}
	\]
	which is the perception-lag row of Table~\ref{tab:coefficient-ledger}.

	(1) Since $b_{0}=a_{2}=0$ in this row,
	\[
		p=a_{0}^{2}-2a_{1}=p_{2}^{2}+(d+k_{3})^{2}>0,
		\qquad n=a_{2}^{2}-b_{2}^{2}=-b_{2}^{2}<0.
	\]
	Whatever the sign of $q$, the coefficient sequence $1,p,q,n$ of $\mathcal H$ has exactly one sign change.
	Descartes' rule of signs therefore gives $\Xi=1$, so the positive root $\xi_{0}$ is unique and simple.
	Because $\mathcal H(0)=n<0$, the polynomial changes sign from negative to positive at $\xi_{0}$; hence $\mathcal H^{\prime}(\xi_{0})>0$.

	(2)--(4) It remains to verify the hypotheses of Theorem~\ref{thm:delay-hopf} under the standing parameter assumptions.
	The only zero of $\mathcal Q(\lambda)=p_{2}M(\lambda+k_{3})$ is $-k_{3}<0$, so $\mathcal P$ and $\mathcal Q$ have no common zero on the imaginary axis.
	Moreover, $a_{2}+b_{2}=p_{2}k_{3}M>0$, so $\lambda=0$ is not a characteristic root.
	Proposition~\ref{prop:imaginary-roots} then gives the pair $\pm i\omega_{0}$ at $\tau=\tau_{0}$.
	Since every nonzero imaginary characteristic root yields a positive root of $\mathcal H$, $\Xi=1$ makes this the only characteristic-root pair on the imaginary axis.
	Lemma~\ref{lem:transversality} gives the rightward transversal crossing, and Theorem~\ref{thm:delay-hopf} establishes the Hopf point and stability on $[0,\tau_0)$.
	Because $\xi_0$ is the only positive root of $\mathcal H$, every subsequent imaginary-axis crossing occurs at one of the delays $\tau_0^{(k)}(\xi_0)$.
	The sign $\mathcal H'(\xi_0)>0$ makes every such crossing rightward, so no characteristic root crosses back into the left half-plane.
	Once the first crossing places characteristic roots in the right half-plane, later crossings can only add unstable roots; hence $F$ cannot restabilise.
\end{proof}

\subsection{The Visitor-Delayed Model}\label{ssec:hopf}

Among the four infeasible placements, we analyse $(\tau_{1},\tau_{2},\tau_{3})=(\tau,0,\tau)$, the visitor-delayed model:
\begin{equation}\label{eq:visitor-delayed}
	\begin{aligned}
		 & \dot V=v_{1}V(t-\tau)-v_{2}V(t-\tau)P(t)+v_{3}K(t), \\
		 & \dot P=p_{1}V(t)-p_{2}P(t),                          \\
		 & \dot K=k_{1}V(t-\tau)-k_{3}K(t).
	\end{aligned}
\end{equation}
This placement delays visitation-driven growth, reinvestment, and the visitor factor in the deterrence term, making it the closest to the visitor-delay structure of earlier delayed-tourism models \cite{russu2009,he2022,he2023}.
Although it is infeasible, local analysis near the interior equilibrium $F$ remains meaningful; Remark~\ref{rem:local-positivity} clarifies the scope of this interpretation.

\begin{proposition}\label{prop:visitor-hopf}
	For the visitor-delayed model, either $\Xi=0$, in which case $F$ is locally asymptotically stable for every $\tau\ge0$, or $\Xi=2$.
	In the latter case, if the transversality and nonresonance hypotheses of Theorem~\ref{thm:delay-hopf} hold at the critical delay $\tau_{0}$, then $F$ is locally asymptotically stable for $\tau\in[0,\tau_{0})$ and undergoes a Hopf bifurcation at $\tau=\tau_{0}$.
\end{proposition}
\begin{proof}
	Linearising \eqref{eq:visitor-delayed} at $F$ gives the visitor-delayed row of Table~\ref{tab:coefficient-ledger}; in particular,
	\[
		\mathcal Q(\lambda)=d\,\lambda(\lambda+p_{2}),
		\qquad d>0.
	\]
	Since $b_{2}=0$, we have $n=a_{2}^{2}>0$, so $\Xi$ is even by \eqref{eq:n-sign}.
	Because the cubic $\mathcal H$ has at most three positive roots, $\Xi\in\{0,2\}$.

	The zeros of $\mathcal Q$ are $0$ and $-p_{2}$, while $\mathcal P(0)=a_{2}>0$; hence $\mathcal P$ and $\mathcal Q$ have no common zero on the imaginary axis.
	Moreover, $\mathcal D(0,\tau)=a_{2}+b_{2}=p_{2}k_{3}M>0$, so $\lambda=0$ is never a characteristic root.
	The result now follows from Theorem~\ref{thm:delay-hopf}.
\end{proof}

Thus Hopf bifurcation is parameter-dependent in this model.
Since $n>0$, the case $\Xi=2$ requires $p<0$ or $q<0$ in \eqref{eq:cubic-coefficients}, and the signs of these coefficients depend on the model parameters.
Section~\ref{sec:numerics} presents a parameter set for which $\Xi=2$.

\begin{remark}\label{rem:local-positivity}
	Infeasibility means that some nonnegative histories leave the nonnegative orthant, not that all do.
	Because $F$ lies in the interior of the orthant, local asymptotic stability ensures that solutions starting sufficiently close to $F$ remain positive for $\tau<\tau_{0}$.
	For a supercritical bifurcation, the periodic orbits shrink to $F$ as $\tau\downarrow\tau_{0}$ and are therefore positive sufficiently close to onset.
	Outside this local regime, including $\tau>\tau_{0}$ in the subcritical case, the analysis gives no positivity guarantee; Section~\ref{sec:numerics} locates boundary contact numerically.
\end{remark}

\section{Numerical Results and Placement Comparisons}\label{sec:numerics}

Numerical computations illustrate both Hopf directions of the perception-lag model \eqref{eq:perception} and contrast it with the delay-independently stable pollution-delayed placement $(0,\tau,\tau)$.
The visitor-delayed comparison illustrates how infeasibility affects the interpretation of the model's local Hopf bifurcation.
Table~\ref{tab:positivity} summarises the simulation outcomes.

\subsection{Numerical Protocol and Scenarios}

\paragraph{Nondimensional form and admissible groups.}
Although the placement family \eqref{eq:family} carries seven dimensional parameters, its scaled dynamics depend on only three dimensionless groups and the rescaled delays $\sigma_i$.
Rescaling the states by their equilibrium values, $\hat V=V/V_{\infty}$, $\hat P=P/P_{\infty}$, and $\hat K=K/K_{\infty}$, and setting $s=p_{2}t$ and $\sigma_i=p_2\tau_i$, we use \eqref{eq:equilibrium-formulas} and Remark~\ref{rem:active} to reduce the family to
\begin{equation}\label{eq:nondim}
	\begin{aligned}
		\hat V'(s) & = \rho\,\hat V(s-\sigma_3)-\mu\,\hat V(s-\sigma_1)\,\hat P(s-\sigma_2)+(\mu-\rho)\,\hat K(s), \\
		\hat P'(s) & = \hat V(s)-\hat P(s),                                                                                      \\
		\hat K'(s) & = \kappa\bigl(\hat V(s-\sigma_3)-\hat K(s)\bigr).
	\end{aligned}
\end{equation}
The equilibrium is $(\hat V,\hat P,\hat K)=(1,1,1)$, and
\[
	\rho=\frac{v_{1}}{p_{2}},\qquad
	\mu=\frac{v_{2}P_{\infty}}{p_{2}},\qquad
	\kappa=\frac{k_{3}}{p_{2}}
\]
measure visitor growth, deterrence, and capital depreciation against natural pollution decay.
By Remark~\ref{rem:active}, the \emph{capital-feedback contribution} $\mu-\rho=v_{3}k_{1}/(k_{3}p_{2})$ is positive, so $\mu>\rho$ is a consequence of the model assumptions rather than an extra condition.
Conversely, every triple with $\rho,\kappa>0$ and $\mu>\rho$ arises from strictly positive parameters; for example, take
\[
	v_{1}=\rho,\qquad k_{3}=\kappa,\qquad k_{1}=\kappa(\mu-\rho),\qquad v_{2}=v_{3}=p_{1}=p_{2}=1.
\]
The admissible region is therefore
\begin{equation}\label{eq:admissible}
	\mathcal G=\{(\rho,\mu,\kappa):\rho>0,\ \kappa>0,\ \mu>\rho\}.
\end{equation}

The three scenarios in Table~\ref{tab:scenarios} differ only in $(v_{1},v_{3})$.
They isolate the qualitative regimes of interest rather than represent a particular site.
Scenarios~1 and 2 realise the two Hopf directions of the perception-lag model, while Scenario~3 supplies the local Hopf and infeasibility comparison for the visitor-delayed model.
Relative to Scenario~1, Scenario~2 increases both the intrinsic visitor-growth group $\rho$ and the capital-feedback contribution $\mu-\rho$.
All integrations begin from the constant history
\begin{equation}\label{eq:history}
	\phi(t)=F+\varepsilon\,(1,-1,1),\qquad t\in[-\tau,0].
\end{equation}

We use the method of steps with SciPy's adaptive DOP853 integrator \cite{virtanen2020}, restarting at each multiple of the delay so that propagated history discontinuities coincide with step-interval boundaries.
The delayed state is evaluated from the solver's dense output on the preceding delay interval.
The nominal tolerances are $\mathtt{rtol}=10^{-9}$ and $\mathtt{atol}=10^{-11}$, with maximum internal step $0.25$; the plotted and diagnostic output is sampled every $0.01$, independently of the adaptive integration steps.
Repeating the principal checks with $\mathtt{rtol}=10^{-11}$, $\mathtt{atol}=10^{-13}$, and maximum step $0.125$ leaves the reported near-onset period and amplitude unchanged to the displayed precision, reproduces the Scenario~2 trajectory through $t=650$, and changes the visitor-delayed orthant-exit time by less than $10^{-9}$.

To accommodate slow transients near $\tau_{0}$, runs measuring settled amplitudes or convergence continue to $T=6000$--$16000$, with diagnostics sampled over the final $15\%$.
Above-critical Scenario~2 diagnostics use the finite horizons stated below; runs used to locate loss of feasibility continue until the first boundary crossing.
For trajectories leaving the nonnegative orthant, we report
\[
	t_{\mathrm{exit}}=\inf\{t\ge0:(V,P,K)(t)\notin[0,\infty)^{3}\}.
\]

As an independent check, DDE-BifTool v3.1.1 \cite{sieber2014}, used with numerical derivatives, recovers the perception-lag Hopf points of Scenarios~1 and 2 and matches the reported $(\omega_{0},\tau_{0})$.
Its first Lyapunov coefficient $L_1$ is negative in Scenario~1 and equals $+5.98\times10^{-4}$ in Scenario~2, agreeing in sign with $\operatorname{Re}c_1(0)$ in Table~\ref{tab:roots}.
Only the signs are compared because the normalisation of $L_1$ differs from that of $c_1(0)$ in \eqref{mukx}.

\begin{table}[htbp]
	\centering
	\footnotesize
	\begin{tabular}{@{}lccc@{}}
		\toprule
		Scenario & $(v_{1},v_{3})$ & $(\rho,\mu,\kappa)$ & $F=(V_{\infty},P_{\infty},K_{\infty})$ \\
		\midrule
		1 & $(2,\,1.5)$ & $(10,\,12.1429,\,3.5)$ & $(0.7359,\,8.0952,\,0.2103)$ \\
		2 & $(5,\,25)$ & $(25,\,60.7143,\,3.5)$ & $(3.6797,\,40.4762,\,1.0513)$ \\
		3 & $(2,\,1.1)$ & $(10,\,11.5714,\,3.5)$ & $(0.7013,\,7.7143,\,0.2004)$ \\
		\bottomrule
	\end{tabular}
	\caption{Parameter scenarios, dimensionless groups, and positive equilibria.
	The scenarios share $v_{2}=0.3$, $p_{1}=2.2$, $p_{2}=0.2$, $k_{1}=0.2$, and $k_{3}=0.7$ and differ only in $(v_{1},v_{3})$.}
	\label{tab:scenarios}
\end{table}

\begin{table}[htbp]
	\centering
	\footnotesize
	\begin{tabular}{@{}lcccccc@{}}
		\toprule
		Scenario & $\omega_{0}$ & $\tau_{0}$ & $\operatorname{Re}\lambda'(\tau_{0})$ & $\operatorname{Re}c_1(0)$ & $\mu_2$ & $|\mathcal D(i\omega_{0},\tau_{0})|$ \\
		\midrule
		1 & $0.5724$ & $0.9647$ & $+0.160$ & $-0.0843$ & $+0.547$ & $4.2\times10^{-17}$ \\
		2 & $0.4958$ & $1.8889$ & $+0.121$ & $+0.0012$ & $-0.00543$ & $8.1\times10^{-16}$ \\
		\bottomrule
	\end{tabular}
	\caption{Computed Hopf quantities for the perception-lag model (Theorem~\ref{thm:perception-hopf}) in Scenarios~1 and 2.
	The reported $\operatorname{Re}\lambda'(\tau_{0})>0$ confirms the transversal crossing guaranteed by the theorem, and $\mu_2=-\operatorname{Re}c_1(0)/(\tau_{0}\operatorname{Re}\lambda'(\tau_{0}))$.
	The final column is the residual of the characteristic function $\mathcal D$ of \eqref{eq:quasipolynomial} at the reported Hopf point.}
	\label{tab:roots}
\end{table}

\begin{table}[htbp]
	\centering
	\footnotesize
	\begin{tabular}{@{}llcc>{\raggedright\arraybackslash}p{5.0cm}@{}}
		\toprule
		Placement & Scenario & $\tau$ & $\tau/\tau_{0}$ & Outcome \\
		\midrule
		Perception-lag & 1 & $0.90$ & $0.93$ & converges to $F$ \\
		Perception-lag & 1 & $1.05$ & $1.09$ & stable positive periodic orbit \\
		Pollution-delayed & 1 & $1.05$ & --- & converges to $F$ \\
		Perception-lag & 2 & $1.7944$ & $0.95$ & converges to $F$ \\
		Perception-lag & 2 & $1.9266$ & $1.02$ & growing oscillations; nonnegative through $t=650$ \\
		Visitor-delayed & 3 & $3.4843$ & $0.85$ & converges to $F$ \\
		Visitor-delayed & 3 & $4.1812$ & $1.02$ & exits the nonnegative orthant at $t_{\mathrm{exit}}\approx342$ \\
		\bottomrule
	\end{tabular}
	\caption{Simulation outcomes across the three delay placements.
	All runs use the history \eqref{eq:history} with $\varepsilon=0.10$, except the Scenario~2 above-critical run ($\varepsilon=0.05$).
	The visitor-delayed critical delay in Scenario~3 is $\tau_{0}=4.0992$; those in Scenarios~1 and~2 are reported in Table~\ref{tab:roots}.
	The dash marks the pollution-delayed placement, which has no critical delay (Theorem~\ref{thm:stable-feasible}).}
	\label{tab:positivity}
\end{table}

\subsection{Perception-Lag Scenario 1: Supercritical Hopf}\label{ssec:perception-supercritical}

For Scenario~1, $\tau_{0}=0.9647$ and $\operatorname{Re}c_1(0)=-0.0843$, so the Hopf bifurcation is supercritical.
Orbitally asymptotically stable periodic solutions emerge for $\tau>\tau_{0}$, with linearised period $2\pi/\omega_{0}=10.98$ at onset.

Figure~\ref{fig:perception} contrasts the time series below and above the critical delay.
The perturbation decays at $\tau=0.90$ and approaches a positive periodic orbit at $\tau=1.05$.
Using the Scenario~1 parameters and the same delay, the pollution-delayed placement $(0,\tau,\tau)$ converges back to $F$ (Table~\ref{tab:positivity}, Theorem~\ref{thm:stable-feasible}).
This contrast shows that the oscillation is due to delay placement rather than delay magnitude.

\begin{figure}[htbp]
	\centering
	\begin{subfigure}[b]{0.495\textwidth}
		\centering
		\includegraphics[width=\linewidth]{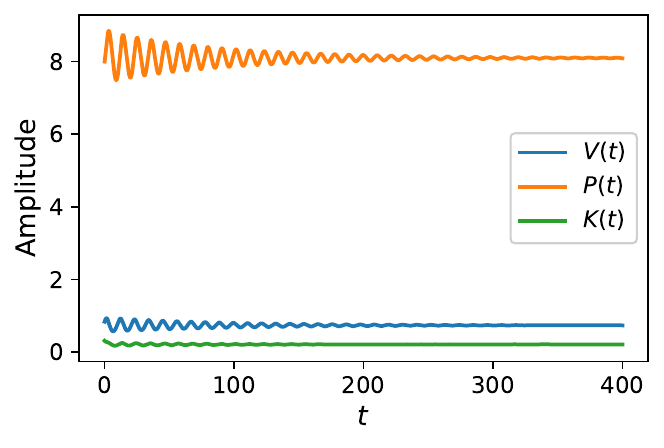}
		\caption{$\tau=0.90$, $\varepsilon=0.10$}
	\end{subfigure}
	\hfill
	\begin{subfigure}[b]{0.495\textwidth}
		\centering
		\includegraphics[width=\linewidth]{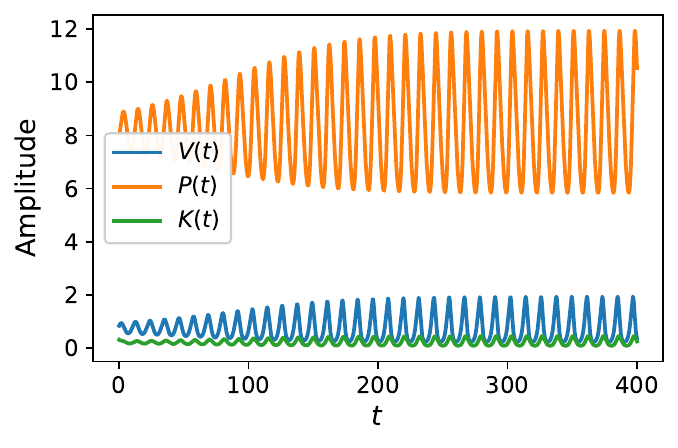}
		\caption{$\tau=1.05$, $\varepsilon=0.10$}
	\end{subfigure}
	\caption{Scenario~1 for the perception-lag model \eqref{eq:perception}.
		At $\tau=0.90<\tau_{0}$, the perturbation decays; at $\tau=1.05>\tau_{0}$, the solution approaches an orbitally asymptotically stable positive periodic orbit.}
	\label{fig:perception}
\end{figure}

Figure~\ref{fig:perception-amplitude} shows the post-transient amplitude $\max V-\min V$ as a function of $\tau/\tau_{0}$.
The amplitude grows continuously from zero with the $\sqrt{\tau-\tau_{0}}$ scaling expected for a supercritical Hopf branch.
Relative to the amplitude at $\tau/\tau_{0}=1.005$, the measured ratios at $1.02$ and $1.05$ are $2.02$ and $3.27$, compared with the predicted values $\sqrt{4}$ and $\sqrt{10}$.
At $\tau/\tau_{0}=1.005$, the measured period is $11.02$, close to the linearised value $10.98$.

\begin{figure}[htbp]
	\centering
	\includegraphics[width=0.64\linewidth]{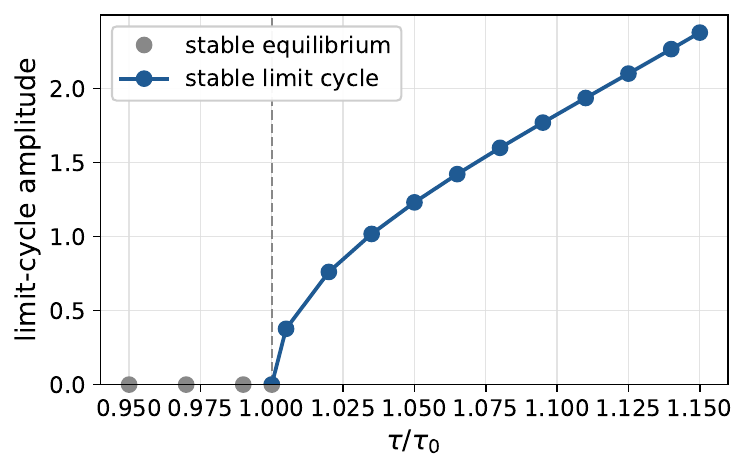}
	\caption{Visitor amplitude $\max V-\min V$, measured over the final $15\%$ of each simulation, for Scenario~1 of the perception-lag model \eqref{eq:perception}.
	The amplitude grows continuously from zero as $\tau/\tau_{0}$ increases above one.}
	\label{fig:perception-amplitude}
\end{figure}

\Needspace{8\baselineskip}
\subsection{Perception-Lag Scenario 2: Subcritical Hopf}

For Scenario~2, $\tau_{0}=1.8889$ and $\operatorname{Re}c_1(0)=+0.0012$, so the Hopf bifurcation is subcritical.
The reported value of $\operatorname{Re}c_1(0)$ is small under the normalisation of Appendix~\ref{app:cm}.
Its magnitude depends on the normal-form normalisation, whereas its sign---and hence the bifurcation classification---does not.
The bifurcating periodic branch is orbitally unstable and lies on the locally stable equilibrium side $\tau<\tau_{0}$.
Continuation shows that this branch need not remain small away from onset (Figure~\ref{fig:perception-sub-branch}).

\paragraph{Perturbation threshold below onset.}
Along the perturbation direction $(1,-1,1)$ in \eqref{eq:history}, the tested histories bracket a finite distance from $F$ across which the outcome changes from return to departure.
This estimated basin boundary is consistent with the unstable periodic branch.
At $\tau=0.98\tau_{0}$, the run with $\varepsilon=1.5$ converges to $F$, while the run with $\varepsilon=2.0$ leaves its neighbourhood.
At $\tau=0.995\tau_{0}$, the corresponding runs are $\varepsilon=0.5$ and $\varepsilon=1.0$.
The normal form predicts that the threshold decreases in proportion to $\sqrt{\tau_{0}-\tau}$, giving a factor of $\sqrt{0.02/0.005}=2$ between these delays, consistent with the observed brackets.

\paragraph{Continued unstable branch.}
Direct continuation with DDE-BifTool confirms that the periodic branch extends into $\tau<\tau_{0}$ and is unstable.
Every computed nontrivial orbit on $0.95\tau_{0}\lesssim\tau<\tau_{0}$ has one nontrivial Floquet multiplier outside the unit circle, indicating one unstable direction (Figure~\ref{fig:perception-sub-branch}).
Although the branch reaches a large visitor amplitude away from onset, with $\max V-\min V\approx23$ at $0.95\tau_{0}$, the continued orbits remain strictly inside the nonnegative orthant over the computed interval.
Their minimum values are $\min V\approx0.65$, $\min P\approx24.5$, and $\min K\approx0.23$, consistent with the forward invariance of the perception-lag model.

\paragraph{Departure above onset.}
None of the simulated trajectories settled onto a small periodic orbit over the reported finite horizons.
The first time at which $|V-V_{\infty}|>1$ is $t=345$ for $1.02\tau_{0}$ and $t=1297$ for $1.005\tau_{0}$, consistent with increasingly slow departure as $\tau\downarrow\tau_{0}$.
The departing trajectories develop rapidly growing oscillations; the tighter-tolerance adaptive run reproduces this growth through $t=650$ and gives the same maximum visitor value, $12.456$, to three decimal places.
Finite-time blow-up is excluded by Lemma~\ref{lem:global-existence}, but this does not imply boundedness as $t\to\infty$ or convergence to an attractor.
The long-time behaviour, including possible convergence to a large-amplitude periodic orbit, therefore remains unresolved.

\begin{figure}[htbp]
	\centering
	\includegraphics[width=0.60\linewidth]{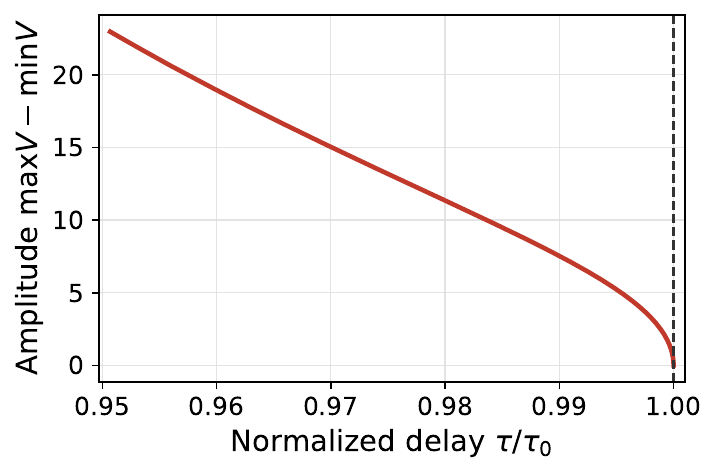}
	\caption{Scenario~2 unstable periodic branch continued from the subcritical Hopf point with DDE-BifTool.
	The amplitude $\max V-\min V$ tends to zero as $\tau\uparrow\tau_{0}$, and all computed nontrivial branch points have one unstable Floquet multiplier.}
	\label{fig:perception-sub-branch}
\end{figure}

\begin{figure}[htbp]
	\centering
	\begin{subfigure}[b]{0.495\textwidth}
		\centering
		\includegraphics[width=\linewidth]{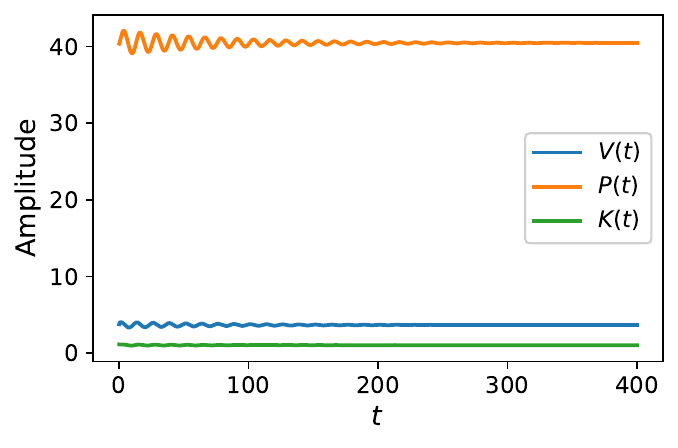}
		\caption{$\tau=0.95\tau_{0}=1.7944$, $\varepsilon=0.10$}
	\end{subfigure}
	\hfill
	\begin{subfigure}[b]{0.495\textwidth}
		\centering
		\includegraphics[width=\linewidth]{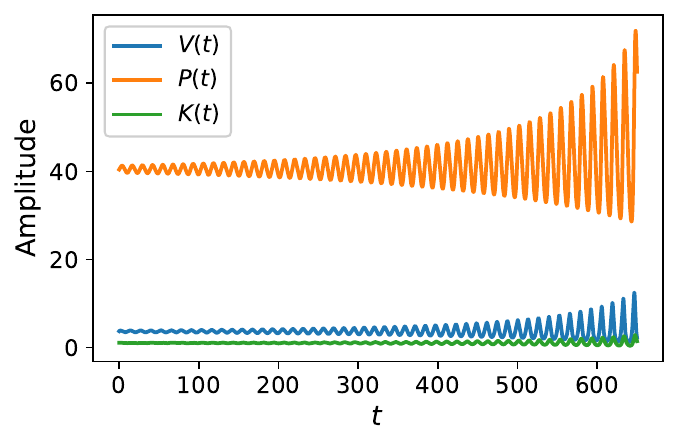}
		\caption{$\tau=1.02\tau_{0}=1.9266$, $\varepsilon=0.05$}
	\end{subfigure}
	\caption{Scenario~2 for the perception-lag model \eqref{eq:perception}.
	At $\tau=0.95\tau_{0}$, the perturbation decays; at $\tau=1.02\tau_{0}$, the oscillatory trajectory grows away from the equilibrium while remaining nonnegative through the plotted horizon $t=650$.}
	\label{fig:perception-sub}
\end{figure}

\FloatBarrier
\subsection{Parameter Dependence of the Perception-Lag Hopf Direction}\label{ssec:sensitivity}

Figure~\ref{fig:perception-sensitivity} maps the sign of $\operatorname{Re}c_1(0)$ on $120\times120$ grids over $\rho\in[5,80]$ and $\mu\in[6,200]$ at $\kappa=1.75$, $3.5$, and $7$.
These slices halve, retain, and double the Scenario~1--2 value $\kappa=3.5$.
Points with $\mu\le\rho$ are inadmissible and shown by the hatched wedge.
Because Theorem~\ref{thm:perception-hopf} guarantees a Hopf point at every admissible parameter, the map classifies direction rather than existence.

At $\kappa=1.75$, all sampled admissible grid points in the displayed window are supercritical, and no zero contour is detected.
At $\kappa=3.5$, the numerically computed zero contour separating the supercritical and subcritical regions rises from $\mu\approx39$ near $\rho\approx13$ to $\mu\approx114$ at $\rho=80$.
At $\kappa=7$, the contour shifts downward and the displayed subcritical region expands.
Scenario~1, at $(10,12.14)$, lies in the supercritical region; Scenario~2, at $(25,60.71)$, lies in the subcritical region.

Throughout the plotted subcritical region, $\operatorname{Re}c_1(0)\lesssim0.002$.
This small magnitude motivates the independent DDE-BifTool sign check and the qualitative perturbation-threshold consistency check in Scenario~2.
The Scenario~2 bifurcation direction is not close to a sign reversal: along $\rho=25$, the boundary occurs near $\mu\approx47$, whereas $\mu=60.71$ lies near the maximum of $\operatorname{Re}c_1(0)$ on that line.

Because the sign of $\operatorname{Re}c_1(0)$ is determined by $(\rho,\mu,\kappa)$, the supercritical--subcritical distinction is unchanged when all seven dimensional coefficients are multiplied by the same positive constant.
The three slices show that bifurcation direction depends materially on $\kappa$; they do not constitute a complete three-dimensional classification.

\begin{figure}[htbp]
	\centering
	\includegraphics[width=\linewidth]{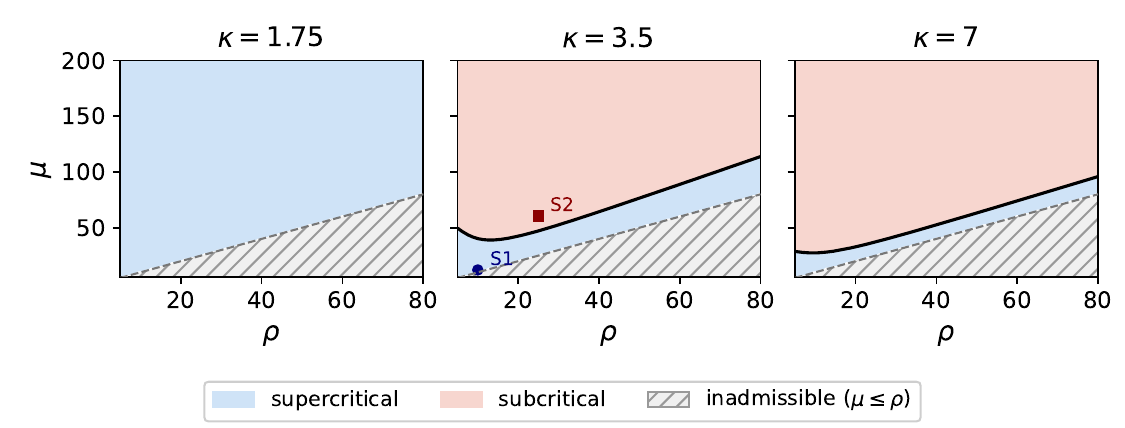}
	\caption{Sign of $\operatorname{Re}c_1(0)$ for the perception-lag model over the $(\rho,\mu)$ plane at the three $\kappa$ values shown.
	The hatched wedge $\mu\le\rho$ lies outside the admissible region $\mathcal G$; the Scenario~1 (S1) and Scenario~2 (S2) markers appear on their $\kappa=3.5$ slice.}
	\label{fig:perception-sensitivity}
\end{figure}

\FloatBarrier
\subsection{The Visitor-Delayed Comparison}

The final comparison uses the visitor-delayed model, whose nonnegative orthant is not forward invariant.
Its bifurcating orbit may therefore lose positivity beyond a neighbourhood of the equilibrium (Proposition~\ref{prop:visitor-hopf}, Remark~\ref{rem:local-positivity}).
For Scenario~3, the critical delay is $\tau_{0}=4.0992$.
The auxiliary cubic has two positive roots; the first gives $\mathcal H'(\xi_0)=0.1597$ and $\operatorname{Re}\lambda'(\tau_0)=0.0204>0$, while the other first reaches the imaginary axis at $\tau=6.4817$.
Thus the critical pair is simple, transversal, and nonresonant; Appendix~\ref{app:cm-visitor} gives $\operatorname{Re}c_1(0)=0.695$, confirming a subcritical Hopf bifurcation.
Figure~\ref{fig:visitor-delayed-comparison} shows decay below this threshold and exit from the nonnegative orthant for $\tau=4.1812$, at $t_{\mathrm{exit}}=341.635$; the tighter-tolerance run agrees to the reported precision.

The same qualitative contrast appears in supercritical and subcritical examples, although the subcritical examples use different parameter scenarios.
At the Scenario~1 parameter set, both placements bifurcate supercritically: the perception-lag periodic orbit remains positive over the tested range through $1.15\,\tau_{0}$, whereas the visitor-delayed periodic orbit, with $\tau_{0}=3.1380$, reaches the boundary of the nonnegative orthant near $1.07\,\tau_{0}$, as located by a delay sweep in the reproducibility scripts.
For the subcritical examples---Scenario~2 for the perception-lag model and Scenario~3 for the visitor-delayed model---the perception-lag trajectory remains positive through its reported finite horizon as it grows away from the equilibrium, whereas the visitor-delayed trajectory exits the nonnegative orthant.

\begin{figure}[htbp]
	\centering
	\begin{subfigure}[b]{0.495\textwidth}
		\centering
		\includegraphics[width=\linewidth]{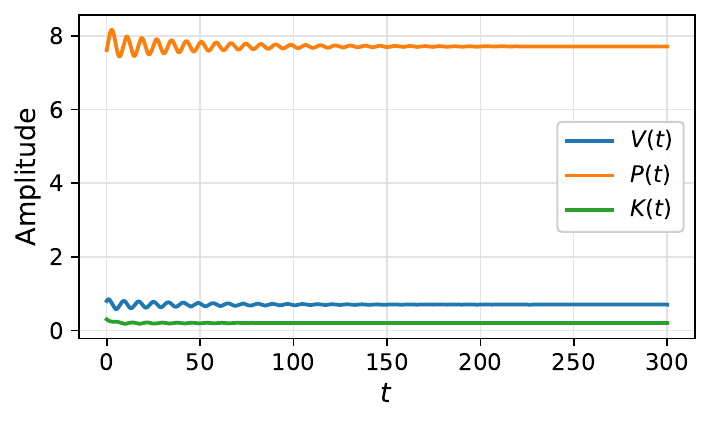}
		\caption{$\tau=3.4843$, $\varepsilon=0.10$}
	\end{subfigure}
	\hfill
	\begin{subfigure}[b]{0.495\textwidth}
		\centering
		\includegraphics[width=\linewidth]{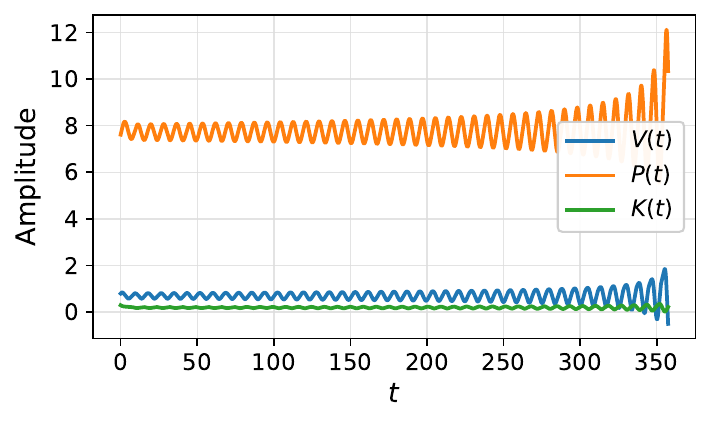}
		\caption{$\tau=4.1812$, $\varepsilon=0.10$}
	\end{subfigure}
	\caption{Visitor-delayed subcritical comparison (Scenario~3) for \eqref{eq:visitor-delayed}.
	At $\tau=0.85\tau_{0}$, the perturbation decays; at $\tau=1.02\tau_{0}$, the trajectory exits the nonnegative orthant at $t_{\mathrm{exit}}=341.635$, stable under the tighter adaptive tolerances.}
	\label{fig:visitor-delayed-comparison}
\end{figure}

\FloatBarrier
\section{Conclusion}\label{sec:conclusion}

We classified the feasibility of all eight common-delay placements and the local stability of the positive equilibrium in the four feasible cases.
For all positive parameters and every $\tau>0$, feasibility holds exactly when the visitor factor in the deterrence term is current, $\tau_{1}=0$.
Every nonnegative history for these four placements generates a unique global nonnegative solution.
Three have a delay-independently stable positive equilibrium; in the perception-lag model, the equilibrium instead loses stability at a finite critical delay for every set of positive parameters and never restabilises as the delay increases.

The center-manifold reduction and numerical results exhibit both bifurcation directions: Scenario~1 has an orbitally asymptotically stable supercritical branch, whereas Scenario~2 has an orbitally unstable subcritical branch.
At the lower $\kappa$ slice, all sampled admissible points in the displayed window are supercritical; the baseline and higher slices contain sampled regions of both directions.
Thus delay placement determines not only feasibility across the eight variants but also whether increasing the delay can destabilise the positive equilibrium among the feasible variants.

The visitor-delayed model $(\tau,0,\tau)$, closest to earlier delayed-tourism models, is infeasible even though it can undergo a local Hopf bifurcation.
Its nonnegative interpretation is therefore guaranteed only locally near the positive equilibrium; particular solutions may remain positive farther away, but forward invariance is lost.
This distinction shows why feasibility must be established before a local bifurcation is given a physical interpretation.
The numerical scenarios illustrate qualitative regimes rather than a calibrated site.

Beyond the specific model, a single structural choice---the slot a delay occupies in a feedback loop---simultaneously governs two properties usually studied separately: whether the nonnegative cone stays invariant and whether the delay can destabilise the equilibrium at all.
When it does, the direction of the onset is fixed not by the placement alone but jointly with the dimensionless parameter groups.

The main analytical priorities are to continue the periodic branches and locate possible folds or termination mechanisms, estimate basin boundaries in the subcritical case, classify the Hopf direction throughout the full three-dimensional group space, and analyse independent perception and response delays in the feasible multi-delay family.
Applied extensions could replace the linear pollution law with calibrated nonlinear recovery and make investment or environmental protection endogenous \cite{cazzavillan2001,menuet2024}, allowing perception, policy, and investment delays to be studied in sustainable or intergenerational development paths \cite{becker1982,martinet2022}.

\subsection*{Code availability}
The supplementary archive contains the Python scripts used to compute the Hopf quantities, run the numerical integrations, generate the figures, and reproduce the tabulated values.
The accompanying README lists the required software, execution commands, and the script associated with each result.
The command \texttt{python3~code/reproduce\_all.py} regenerates the Python-based results, including the period, perturbation-threshold, and departure-time diagnostics reported in Section~\ref{sec:numerics}.
Wolfram and DDE-BifTool implementations provide independent checks of the characteristic roots and the Hopf normal-form coefficients.
The archive is included with the submission and released under the MIT License.
A versioned public copy will be deposited on Zenodo upon acceptance, and its DOI will be added to the final article.

\subsection*{Declarations}
\textbf{Funding.} This research did not receive any specific grant from funding agencies in the public, commercial, or not-for-profit sectors.

\noindent\textbf{Competing interests.} The authors declare that they have no known competing financial interests or personal relationships that could have appeared to influence the work reported in this paper.

\Needspace{5\baselineskip}
\noindent\textbf{CRediT authorship contribution statement.} \textbf{Taylan Şengül:} Conceptualization, Methodology, Supervision, Writing -- review \& editing. \textbf{Bünyamin Kurtkaya:} Formal analysis, Software, Writing -- original draft.

\noindent\textbf{Data availability.} No external datasets were used; the numerical parameter assignments and histories are reported in the manuscript, and the reproduction scripts are included in the supplementary archive.

\appendix
\section{Center-Manifold Reduction of the Hopf Bifurcation}\label{app:cm}
\subsection{Functional Differential Equation Form}
The reduction follows the center-manifold procedure of Hassard, Kazarinoff, and Wan \cite{hassard1981}.
The calculation first identifies the current and delayed linear matrices and the quadratic nonlinearity, and then constructs the critical eigenvectors and their normalisation.
The equations for $W_{20}$ and $W_{11}$ provide the second-order center-manifold terms needed to evaluate $g_{21}$ and hence $c_1(0)$.
The routine \path{code/general_cm.py} implements these steps for both analysed models and reproduces the coefficients reported in Section~\ref{sec:numerics}.
For the perception-lag model \eqref{eq:perception}, set
\begin{equation}\label{eq:pcm-shift}
	x(t)=\bigl(x_1(t),x_2(t),x_3(t)\bigr)^{T}
	=\bigl(V(t)-V_{\infty},P(t)-P_{\infty},K(t)-K_{\infty}\bigr)^{T}.
\end{equation}
The shifted perception-lag model is
\begin{equation}\label{eq:pcm-shifted-system}
	\begin{aligned}
		\dot x_1(t) & =\left(v_1-v_2P_{\infty}\right)x_1(t)-v_2V_{\infty}x_2(t-\tau)+v_3x_3(t)-v_2x_1(t)x_2(t-\tau), \\
		\dot x_2(t) & =p_1x_1(t)-p_2x_2(t), \\
		\dot x_3(t) & =k_1x_1(t)-k_3x_3(t).
	\end{aligned}
\end{equation}
Let $s$ denote original time and set $t=s/\tau$, so that the delay in the new time variable is one.
Relabel the rescaled state as $x(t)$ and introduce the unfolding parameter
\begin{equation}\label{eq:pcm-rescaling}
	\nu=\tau-\tau_0.
\end{equation}
Then the delay interval is $[-1,0]$, and \eqref{eq:pcm-shifted-system} becomes
\begin{equation}\label{eq:pcm-rescaled-system}
	\begin{aligned}
		\dot x_1(t) & =(\tau_0+\nu)\left[\left(v_1-v_2P_{\infty}\right)x_1(t)-v_2V_{\infty}x_2(t-1)+v_3x_3(t)-v_2x_1(t)x_2(t-1)\right], \\
		\dot x_2(t) & =(\tau_0+\nu)\left[p_1x_1(t)-p_2x_2(t)\right], \\
		\dot x_3(t) & =(\tau_0+\nu)\left[k_1x_1(t)-k_3x_3(t)\right].
	\end{aligned}
\end{equation}
Under this rescaling, an original-time characteristic root $\lambda(\tau)$ becomes $\Lambda(\tau)=\tau\lambda(\tau)$, so $\operatorname{Re}\Lambda'(\tau_0)=\tau_0\operatorname{Re}\lambda'(\tau_0)$ because $\operatorname{Re}\lambda(\tau_0)=0$.
Let $\mathcal X=C([-1,0],\mathbb C^3)$ denote the complexification of the real phase space, and set $x_t(\theta)=x(t+\theta)$.
The linear part is $L_{\nu}\phi=\int_{-1}^{0}d\eta(\theta,\nu)\phi(\theta)$, where
\begin{equation}\label{eq:pcm-B-C}
	B=\begin{pmatrix}
		v_1-v_2P_\infty & 0 & v_3\\
		p_1&-p_2&0\\
		k_1&0&-k_3
	\end{pmatrix},
	\qquad
	C=\begin{pmatrix}
		0&-v_2V_\infty&0\\
		0&0&0\\
		0&0&0
	\end{pmatrix},
\end{equation}
and the Riesz measure is
\begin{equation}\label{eq:pcm-eta}
	d\eta(\theta,\nu)=(\tau_0+\nu)B\,\delta_0(\theta)+(\tau_0+\nu)C\,\delta_{-1}(\theta).
\end{equation}
Here $\delta_a$ denotes the unit point mass at $a$.
The quadratic remainder is
\begin{equation}\label{eq:pcm-nonlinearity}
	f(\nu,\phi)=(\tau_0+\nu)
	\begin{pmatrix}
		-v_2\phi_1(0)\phi_2(-1)\\
		0\\
		0
	\end{pmatrix}.
\end{equation}
Define $\mathcal A(\nu)$ and $R(\nu)$ by
\begin{equation}\label{eq:pcm-operator}
	\mathcal A(\nu)\phi=
	\begin{cases}
		\dfrac{d\phi(\theta)}{d\theta}, & \theta\in[-1,0),\\
		\displaystyle\int_{-1}^{0}d\eta(s,\nu)\phi(s), & \theta=0,
	\end{cases}
	\qquad
	R(\nu)\phi=
	\begin{cases}
		0, & \theta\in[-1,0),\\
		f(\nu,\phi), & \theta=0.
	\end{cases}
\end{equation}
Thus the rescaled retarded equation has the operator form
\begin{equation}\label{eq:pcm-operator-form}
	\dot x_t=\mathcal A(\nu)x_t+R(\nu)x_t.
\end{equation}

\subsection{Spectral Properties of the Operator \texorpdfstring{$\mathcal A$}{A} and Its Adjoint}
At $\nu=0$, write $\mathcal A=\mathcal A(0)$, the infinitesimal generator of the solution semigroup.
The adjoint operator is
\begin{equation}\label{eq:pcm-adjoint}
	\mathcal A^*\psi(s)=
	\begin{cases}
		-\dfrac{d\psi(s)}{ds}, & s\in(0,1],\\
		\displaystyle\int_{-1}^{0}d\eta^{T}(t,0)\psi(-t), & s=0,
	\end{cases}
\end{equation}
with respect to the bilinear form
\begin{equation}\label{eq:pcm-bilinear}
	\langle\psi,\phi\rangle=\bar\psi^{T}(0)\phi(0)-\int_{-1}^{0}\int_{0}^{\theta}\bar\psi^{T}(\xi-\theta)\,d\eta(\theta,0)\,\phi(\xi)\,d\xi.
\end{equation}
Let $\pm i\omega_0$ be the critical pair at $\tau=\tau_0$ and set $\Omega=\omega_0\tau_0$.
The eigenvector of $\mathcal A$ associated with $i\Omega$ is
\begin{equation}\label{eq:pcm-eigvec}
	q(\theta)=
	\begin{pmatrix}
		1\\
		\beta\\
		\gamma
	\end{pmatrix}
	e^{i\Omega\theta},
	\qquad
	\beta=\frac{p_1}{p_2+i\omega_0},
	\qquad
	\gamma=\frac{k_1}{k_3+i\omega_0}.
\end{equation}
The adjoint eigenvector of $\mathcal A^*$ associated with $-i\Omega$ is
\begin{equation}\label{eq:pcm-adjoint-eigvec}
	q^*(s)=D
	\begin{pmatrix}
		1\\
		\beta^*\\
		\gamma^*
	\end{pmatrix}
	e^{i\Omega s},
	\qquad
	\beta^*=-\frac{v_2V_\infty e^{i\Omega}}{p_2-i\omega_0},
	\qquad
	\gamma^*=\frac{v_3}{k_3-i\omega_0}.
\end{equation}
The normalisation $\langle q^*,q\rangle=1$ gives
\begin{equation}\label{eq:pcm-normalization}
	\overline D=
	\left[
		1+\beta\,\overline{\beta^*}+\gamma\,\overline{\gamma^*}
		-\tau_0v_2V_\infty\beta e^{-i\Omega}
	\right]^{-1}.
\end{equation}

\subsection{Center-Manifold Reduction}\label{app:cm-center}
Let $z(t)=\langle q^*,x_t\rangle$ and write
\begin{equation}\label{eq:pcm-W-def}
	W(t,\theta)=x_t(\theta)-z(t)q(\theta)-\bar z(t)\bar q(\theta).
\end{equation}
On the center manifold,
\begin{equation}\label{eq:pcm-W-expansion}
	W(z,\bar z,\theta)=W_{20}(\theta)\frac{z^2}{2}+W_{11}(\theta)z\bar z+W_{02}(\theta)\frac{\bar z^2}{2}+\cdots .
\end{equation}
The reduced equation is
\begin{equation}\label{eq:pcm-reduced}
	\dot z=i\Omega z+g(z,\bar z),
	\qquad
	g(z,\bar z)=g_{20}\frac{z^2}{2}+g_{11}z\bar z+g_{02}\frac{\bar z^2}{2}+g_{21}\frac{z^2\bar z}{2}+\cdots .
\end{equation}
Substitution of $x_t=W+zq+\bar z\bar q$ into \eqref{eq:pcm-nonlinearity} gives
\begin{equation}\label{eq:pcm-g-coeffs}
	\begin{aligned}
		g_{20} &=-2\tau_0v_2\overline D\,\beta e^{-i\Omega},\\
		g_{11} &=-\tau_0v_2\overline D\left(\beta e^{-i\Omega}+\bar\beta e^{i\Omega}\right),\\
		g_{02} &=-2\tau_0v_2\overline D\,\bar\beta e^{i\Omega}.
	\end{aligned}
\end{equation}
For $\theta\in[-1,0)$, the center-manifold invariance equations integrate to
\begin{equation}\label{eq:pcm-W-solutions}
	\begin{aligned}
		W_{20}(\theta)&=\frac{ig_{20}}{\Omega}q(0)e^{i\Omega\theta}+\frac{i\bar g_{02}}{3\Omega}\bar q(0)e^{-i\Omega\theta}+E_1e^{2i\Omega\theta},\\
		W_{11}(\theta)&=-\frac{ig_{11}}{\Omega}q(0)e^{i\Omega\theta}+\frac{i\bar g_{11}}{\Omega}\bar q(0)e^{-i\Omega\theta}+E_2.
	\end{aligned}
\end{equation}
Nonresonance makes $2i\omega_0 I-B-Ce^{-2i\Omega}$ invertible, and the boundary equation for $W_{20}$ gives
\begin{equation}\label{eq:pcm-E1-system}
	\left(2i\omega_0 I-B-Ce^{-2i\Omega}\right)E_1
	=
	2
	\begin{pmatrix}
		-v_2\beta e^{-i\Omega}\\
		0\\
		0
	\end{pmatrix}.
\end{equation}
Set
\begin{equation}\label{eq:pcm-Delta20}
	\Delta_{20}=2i\omega_0+\frac{v_3k_1}{k_3}+v_2V_\infty\frac{p_1e^{-2i\Omega}}{p_2+2i\omega_0}-v_3\frac{k_1}{k_3+2i\omega_0}.
\end{equation}
Solving for $E_1$ gives
\begin{equation}\label{eq:pcm-E1}
	E_1=
	-\frac{2v_2\beta e^{-i\Omega}}{\Delta_{20}}
	\begin{pmatrix}
		1\\
		\dfrac{p_1}{p_2+2i\omega_0}\\
		\dfrac{k_1}{k_3+2i\omega_0}
	\end{pmatrix}.
\end{equation}
For $W_{11}$, set
\begin{equation}\label{eq:pcm-S}
	S=\beta e^{-i\Omega}+\bar\beta e^{i\Omega}.
\end{equation}
Because $\lambda=0$ is excluded, $B+C$ is invertible, and the boundary equation for $W_{11}$ is
\begin{equation}\label{eq:pcm-E2-system}
	(B+C)E_2
	=
	\begin{pmatrix}
		v_2S\\
		0\\
		0
	\end{pmatrix},
\end{equation}
and hence
\begin{equation}\label{eq:pcm-E2}
	E_2=
	-S
	\begin{pmatrix}
		\dfrac{p_2}{V_\infty p_1}\\
		\dfrac{1}{V_\infty}\\
		\dfrac{k_1p_2}{k_3V_\infty p_1}
	\end{pmatrix}.
\end{equation}
Write $W_{jk}^{(\ell)}$ for the $\ell$-th component of $W_{jk}$.
The cubic normal-form coefficient is
\begin{equation}\label{eq:pcm-g21}
	g_{21}
	=
	-2\tau_0v_2\overline D
	\left[
		W_{11}^{(2)}(-1)+\frac{1}{2}W_{20}^{(2)}(-1)
		+\frac{1}{2}W_{20}^{(1)}(0)\bar\beta e^{i\Omega}
		+W_{11}^{(1)}(0)\beta e^{-i\Omega}
	\right].
\end{equation}
The coefficient $c_1(0)$ in \eqref{mukx} is
\begin{equation}\label{eq:pcm-c1}
	c_1(0)=\frac{i}{2\omega_0\tau_0}
	\left(g_{20}g_{11}-2|g_{11}|^2-\frac{1}{3}|g_{02}|^2\right)
	+\frac{1}{2}g_{21}.
\end{equation}

\subsection{Visitor-Delayed Evaluation}\label{app:cm-visitor}

At a Hopf point of any other common-delay placement with a simple critical pair, the preceding reduction applies after replacing the matrices $B$, $C$, and the quadratic term $f$.
For the visitor-delayed model \eqref{eq:visitor-delayed}, these quantities are
\[
	B=
	\begin{pmatrix}
		0&-v_2V_\infty&v_3\\
		p_1&-p_2&0\\
		0&0&-k_3
	\end{pmatrix},
	\qquad
	C=
	\begin{pmatrix}
		v_1-v_2P_\infty&0&0\\
		0&0&0\\
		k_1&0&0
	\end{pmatrix},
\]
and
\[
	f(\nu,\phi)=(\tau_0+\nu)
	\begin{pmatrix}
		-v_2\phi_1(-1)\phi_2(0)\\
		0\\
		0
	\end{pmatrix}.
\]
Repeating the reduction with these data gives $\operatorname{Re}c_1(0)=-0.978$ for the Scenario~1 parameter set and $+0.695$ for Scenario~3, confirming supercritical and subcritical directions, respectively.

\section{Absence of Hopf Bifurcation in the Caraballo Modification}\label{app:caraballo}

Caraballo et al. \cite{caraballo2019} restore invariance of the boundary $E=0$ in \eqref{eq:russu} by multiplying the delayed deterrence by the current environmental quality.
With $q=b-c\eta>0$ and $s=1-\eta>0$, their modified system is
\begin{equation}\label{eq:caraballo}
	\dot V=m_1E+m_2K-aV^2,\qquad
	\dot E=r(\bar P-E)-qV(t-\tau)E,\qquad
	\dot K=sV(t-\tau)-\delta K.
\end{equation}
Although Caraballo et al. do not analyse \eqref{eq:caraballo}, its positive equilibrium is locally asymptotically stable for every delay.

\begin{proposition}\label{prop:caraballo-nohopf}
	If \eqref{eq:caraballo} has a positive equilibrium $(V_{*},E_{*},K_{*})$ with $q,s>0$, then $(V_{*},E_{*},K_{*})$ is locally asymptotically stable for every $\tau\ge0$, and \eqref{eq:caraballo} undergoes no Hopf bifurcation.
\end{proposition}

\begin{proof}
	Write $A=2aV_{*}$, $B=r+qV_{*}$, $C=\delta$, $u=m_1qE_{*}$, and $v=m_2s$, all positive.
	The equilibrium gives $K_{*}=sV_{*}/C$ and, from $\dot V=0$,
	\begin{equation}\label{eq:caraballo-identity}
		aV_{*}=\frac{m_1E_{*}}{V_{*}}+\frac{v}{C}>\frac{v}{C}.
	\end{equation}
	Linearising \eqref{eq:caraballo} at the equilibrium yields the characteristic equation
	\begin{equation}\label{eq:caraballo-char}
		(\lambda+A)(\lambda+B)(\lambda+C)+\bigl[u(\lambda+C)-v(\lambda+B)\bigr]e^{-\lambda\tau}=0 ,
	\end{equation}
	of the form \eqref{eq:quasipolynomial} with $a_0=A+B+C$, $a_1=AB+AC+BC$, $a_2=ABC$, $b_0=0$, $b_1=u-v$, and $b_2=uC-vB$.
	A purely imaginary root $\lambda=i\omega$ requires, by \eqref{eq:cubic-in-z}--\eqref{eq:cubic-coefficients} with $\xi=\omega^{2}$,
	\begin{equation}\label{eq:caraballo-cubic}
		\mathcal H(\xi)=\xi^{3}+(A^{2}+B^{2}+C^{2})\,\xi^{2}+\bigl[A^{2}B^{2}+A^{2}C^{2}+B^{2}C^{2}-(u-v)^{2}\bigr]\xi+\bigl[A^{2}B^{2}C^{2}-(uC-vB)^{2}\bigr]=0 .
	\end{equation}
	From \eqref{eq:caraballo-identity} and $B>qV_{*}$ we obtain $AB>2u$ and $AC>2v$.
	Hence $(u-v)^{2}<\max(AB,AC)^{2}\le A^{2}B^{2}+A^{2}C^{2}$, so the $\xi$-coefficient is positive.
	Also, $ABC>uC$ and $ABC>vB$, which together imply $ABC>|uC-vB|$ and make the constant term positive.
	All coefficients of $\mathcal H$ are therefore positive, giving $\mathcal H(\xi)>0$ for every $\xi\ge0$, so \eqref{eq:caraballo-char} has no nonzero purely imaginary root for any $\tau\ge0$.
	Moreover $\mathcal D(0,\tau)=ABC+uC-vB>0$ for every $\tau\ge0$, so zero is never a characteristic root either.
	At $\tau=0$ the characteristic polynomial $\lambda^{3}+(A+B+C)\lambda^{2}+(AB+AC+BC+u-v)\lambda+(ABC+uC-vB)$ has positive coefficients.
	Moreover,
	\[
		\begin{aligned}
			&(A+B+C)(AB+AC+BC+u-v)-(ABC+uC-vB)\\
			&\quad=(A+B+C)(AB+AC+BC)-ABC+u(A+B)-v(A+C)>0,
		\end{aligned}
	\]
	To see the final inequality explicitly, the term $(A+B+C)(AB+AC+BC)-ABC$ is at least $AC(A+C)$, while \eqref{eq:caraballo-identity} gives $v<AC/2$.
	Therefore $AC(A+C)-v(A+C)>0$, and the remaining term $u(A+B)$ is also positive.
	Thus the Routh--Hurwitz inequality holds, and the equilibrium is locally asymptotically stable.
	Since no characteristic root crosses the imaginary axis as $\tau$ increases, stability persists for all $\tau\ge0$, and no Hopf bifurcation occurs.
\end{proof}

\section*{Declaration of generative AI and AI-assisted technologies in the manuscript preparation process}
During the preparation of this work the authors used Claude (Anthropic) and ChatGPT (OpenAI) in order to improve the language and readability of the manuscript and to assist with drafting and editing.
After using these tools, the authors reviewed and edited the content as needed and take full responsibility for the content of the publication.

\printbibliography

\end{document}